\documentclass[12pt,english]{amsart}
\usepackage[T1]{fontenc}
\usepackage[latin9]{inputenc}
\usepackage{amstext}
\usepackage{amsthm}
\usepackage{amssymb}
\usepackage{cancel}
\usepackage{mathdots}

\makeatletter
\numberwithin{equation}{section}
\numberwithin{figure}{section}
\theoremstyle{plain}
\newtheorem{thm}{\protect\theoremname}
\theoremstyle{remark}
\newtheorem{rem}[thm]{\protect\remarkname}
\theoremstyle{definition}
\newtheorem{example}[thm]{\protect\examplename}
\theoremstyle{remark}
\newtheorem*{rem*}{\protect\remarkname}
\theoremstyle{plain}
\newtheorem{prop}[thm]{\protect\propositionname}
\theoremstyle{plain}
\newtheorem{lem}[thm]{\protect\lemmaname}
\theoremstyle{plain}
\newtheorem{fact}[thm]{\protect\factname}
\newenvironment{lyxcode}
	{\par\begin{list}{}{
		\setlength{\rightmargin}{\leftmargin}
		\setlength{\listparindent}{0pt}
		\raggedright
		\setlength{\itemsep}{0pt}
		\setlength{\parsep}{0pt}
		\normalfont\ttfamily}%
	 \item[]}
	{\end{list}}

\usepackage{hyperref}
\usepackage{dynkin-diagrams}

\AtBeginDocument{
  
}

\makeatother

\usepackage{babel}
\usepackage[style=numeric,backend=bibtex]{biblatex}
\providecommand{\examplename}{Example}
\providecommand{\factname}{Fact}
\providecommand{\lemmaname}{Lemma}
\providecommand{\propositionname}{Proposition}
\providecommand{\remarkname}{Remark}
\providecommand{\theoremname}{Theorem}

\begin{document}
\title{Poles of unramified Degenerate Eisenstein Series}
\email{hegde039@umn.edu}
\address{School of Mathematics, University of Minnesota, Twin-Cities - 55455. }
\author{Devadatta G. Hegde}
\begin{abstract}
We determine the locations and the orders of the poles in the half-plane
$\text{Re}(s)\ge0$ of unramified degenerate Eisenstein series attached
to a maximal proper parabolic subgroup of a split semi-simple linear
algebraic group over a number field. 
\end{abstract}

\maketitle
\tableofcontents{}

\section{Introduction}

Let $\mathfrak{H}$ be the usual upper half-plane and $z=x+iy\in\mathfrak{H}$.
The simplest example of Eisenstein series is 
\[
E_{s}(z)=\frac{1}{2}\sum_{c,d\text{ coprime}}\frac{y^{\frac{1}{2}+\frac{s}{2}}}{|cz+d|^{s+1}},\quad\text{Re}(s)>1
\]
where the sum is over all integer pairs $(c,d)\in\mathbb{Z}^{2}$
with $\text{gcd}(c,d)=1$. The sum converges for $\text{Re}(s)>1$.
The map $s\mapsto E_{s}$ has a meromorphic continuation to $\mathbb{C}$
as a vector-valued map taking values in the space of smooth functions
of uniform moderate growth on $SL_{2}(\mathbb{Z})\backslash\mathfrak{H}$
(see Bernstein-Lapid \cite{Lapid_Bernstein}).

The constant term of $E_{s}$ is
\[
\int_{0}^{1}E_{s}(x+iy)dx=y^{\frac{1}{2}+\frac{s}{2}}+c(s)y^{\frac{1}{2}-\frac{s}{2}},
\]
where 
\[
c(s)=\frac{\xi(s)}{\xi(1+s)},\quad\xi(s)=\pi^{-\frac{s}{2}}\Gamma(s/2)\zeta(s)
\]
The Eisenstein series $E_{s}$ satisfies the functional equation 
\[
E_{s}=c(s)E_{-s}
\]

To motivate the objects appearing in the main theorem \ref{thm:intro}
of this paper and to explain the complementary nature of this result
to those of Langlands, we recall Langlands's landmark computation
of the analogues of function $c(s)$ above appearing in the functional
equation of maximal parabolic \emph{cuspidal} Eisenstein series \cite{Langlands-EulerProduct}.
We follow Casselman's account \cite{Cass_LGroup} and some unpublished
notes of Erez Lapid. Now we use the standard notation in the subject
due to Arthur \cite{Arthur-trace}, recalled in section \ref{sec:Degenerate-Eisenstein-Series}. 

Let $G$ be a split reductive group over a number field $F$. Let
$T\subset B\subset G$ be a maximal $F$-split torus $T$ contained
in a Borel subgroup $B$ defined over $F$. A \emph{standard} parabolic
subgroup is an $F$-parabolic subgroup of $G$ containing $B$. 
\begin{rem}
\label{rem:para-convention}As is common in the subject, in what follows,
a parabolic subgroup is always understood to be a\emph{ proper standard
}parabolic subgroup unless explicitly cautioned otherwise. For example,
a maximal parabolic subgroup means a standard maximal proper parabolic
subgroup. We refer to the standard Levi decomposition of a parabolic
subgroup as \emph{the }Levi decomposition (see I.1.4 of Moeglin-Waldspurger
\cite{MoeglinWaldspurger-book}).
\end{rem}

Let $\mathbb{A}$ be the adele ring of $F$. Let $X(G)$ be the lattice
of $F$-characters of $G$ and 
\[
G(\mathbb{A})^{1}=\bigcap_{\chi\in X(G)}\ker|\chi|
\]
Let $\mathbf{K}$ be the standard maximal compact subgroup of $G(\mathbb{A})$
(see I.1.4 of Moeglin-Waldspurger \cite{MoeglinWaldspurger-book}). 

\subsection{Constant terms and cuspforms}

For a function $f$ on $G(F)\backslash G(\mathbb{A})$ and a parabolic
subgroup $P$ with the unipotent radical $N$, the \emph{constant
term of $f$ along $P$} is 
\[
c_{P}f(g):=\int_{N(F)\backslash N(\mathbb{A})}f(ng)dn
\]
The space of \emph{cuspforms }is
\[
L_{0}^{2}(G(F)\backslash G(\mathbb{A})^{1})
\]
\[
=\left\{ f\in L^{2}(G(F)\backslash G(\mathbb{A})^{1}):c_{P}f=0\ \forall\text{ parabolic subgroups }P\neq G\right\} 
\]

The space $\mathcal{H}_{\mathbf{K}}:=C_{c}^{\infty}(G(\mathbb{A})//\mathbf{K})$
of $\mathbf{K}$-bi-invariant test functions on $G(\mathbb{A})$ is
a \emph{commutative }algebra under convolution. A basic result is
that the space $L_{0}^{2}(G(F)\backslash G(\mathbb{A})^{1})^{\mathbf{K}}$
of $\mathbf{K}$-invariant cuspforms decomposes discretely with respect
to $\mathcal{H}_{\mathbf{K}}$:
\[
L_{0}^{2}(G(F)\backslash G(\mathbb{A})^{1})^{\mathbf{K}}=\hat{\bigoplus_{\chi:\mathcal{H}_{\mathbf{K}}\to\mathbb{C}}}V_{\chi}
\]
where hat $\hat{\cdot}$ on the sum denotes completion and $\mathcal{H}_{\mathbf{K}}$
acts on $V_{\chi}$ by $\chi$. We note that $\dim V_{\chi}<\infty$
for each algebra homomorphism $\chi:\mathcal{H}_{\mathbf{K}}\to\mathbb{C}$.
We refer to a \emph{non-zero} eigenfunction of $\mathcal{H}_{\mathbf{K}}$
in $L_{0}^{2}(G(F)\backslash G(\mathbb{A})^{1})^{\mathbf{K}}$ as
a \emph{strong sense cuspform} on $G(\mathbb{A})$. 

\subsection{Maximal parabolic cuspidal Eisenstein series}

For the rest of this introduction, we assume that $G$ is \emph{semi-simple}
and $P=N\rtimes M$ is the Levi decomposition of a \emph{maximal }parabolic
subgroup $P$ of $G$ (see remark \ref{rem:para-convention}). Let
$\delta_{P}:G(\mathbb{A})\to(0,\infty)$ be the extension to $G(\mathbb{A})$
of the modular character on $P(\mathbb{A})$ using the Iwasawa decomposition
$G(\mathbb{A})=P(\mathbb{A})\mathbf{K}$. Let $\varpi$ be the fundamental
weight corresponding to $P$. We parameterize the space $\mathfrak{a}_{P}^{*}\otimes\mathbb{C}$
by $s\varpi$ for $s\in\mathbb{C}$. 

Let $\varphi$ be a strong sense cuspform on $M(\mathbb{A})$. We
denote the usual extension of $\varphi$ to $G(\mathbb{A})$ using
the Levi-Langlands decomposition 
\[
\varphi(g)=\varphi(m),\ g=nmak\in N(\mathbb{A})M(\mathbb{A})^{1}A_{M}^{+}\mathbf{K}
\]
also by the same letter $\varphi$. Let
\[
\varphi_{s}=\delta_{P}^{\frac{1}{2}}\cdot\varphi\cdot e^{\langle s\varpi,H_{P}(g)\rangle}\in C^{\infty}(G(\mathbb{A}))
\]
The \emph{Eisenstein series }made from \emph{cuspidal data }$\varphi$
is 
\[
E^{P}(s,\varphi,g):=\sum_{\gamma\in P(F)\backslash G(F)}\varphi_{s}(\gamma g)
\]
The sum converges for $\text{Re}(s)\gg0$. It has a meromorphic continuation
to $\mathbb{C}$ \cite{Lapid_Bernstein}.

Let $\chi:\mathcal{H}_{\mathbf{K}\cap M(\mathbb{A})}\to\mathbb{C}$
be the character corresponding to $\varphi$ and let $V_{\chi}$ be
the corresponding eigenspace. Further assume that the closure of the
space generated by $\varphi$ is an \emph{irreducible} cuspidal automorphic
representation $\pi$ of $M(\mathbb{A})$. The Eisenstein series $E^{P}(s,\varphi,g)$
satisfies a functional equation
\begin{equation}
E^{P}(s,\varphi,g)=c(s,\pi)E^{\overline{P}}(-s,\varphi',g)\label{eq:c(s,pi)}
\end{equation}
for some $\varphi'\in V_{\chi}$, where $\overline{P}=\overline{N}\rtimes M$
is the parabolic (non-standard) opposite to $P$. Langlands computed
$c(s,\pi)$ as a ratio of products of $L$-functions in \cite{Langlands-EulerProduct}. 

\subsection{\label{subsec:LanglandsComp}Langlands's computation}

In the Langlands dual group $^{L}G$, there is a corresponding parabolic
subgroup $^{L}P={}^{L}N\rtimes{}^{L}M$ with Levi $^{L}M$ and the
unipotent radical $^{L}N$. The maximal torus $A$ in the \emph{center}
of $^{L}M$ is one dimensional since $^{L}G$ is semi-simple and $^{L}P$
is a maximal parabolic subgroup of $^{L}G$. Consider the eigenspace
decomposition of $^{L}\mathfrak{n}$ under the adjoint action of $A$:
\[
^{L}\mathfrak{n}:=\text{Lie}({}^{L}N)=r_{1}\oplus\cdots\oplus r_{m}
\]
where 
\[
r_{j}=\bigoplus_{\alpha^{\vee}:\langle\varpi,\alpha^{\vee}\rangle=j}{}^{L}\mathfrak{g}_{\alpha^{\vee}}\quad j\ge1
\]
and the sum is over coroots $\alpha^{\vee}$ of $G$ for which $\langle\varpi,\alpha^{\vee}\rangle=j$,
that is if $\beta^{\vee}$ be the simple coroot corresponding to $^{L}P$,
then
\[
\alpha^{\vee}=\cdots+j\beta^{\vee}+\cdots
\]
when $\alpha^{\vee}$ is written as a sum of simple coroots. By a
theorem of Shahidi \cite{Shahidi-Auto-L-func-book}, each $r_{i}$
is an irreducible representation of $^{L}M$. 

Langlands \cite{Langlands-EulerProduct} expressed $c(s,\pi)$ in
equation \ref{eq:c(s,pi)} as a product of ratios of $L$-functions:

\[
c(s,\pi)=\prod_{j=1}^{m}\frac{L(js,\pi,r_{j})}{L(1+js,\pi,r_{j})}
\]
The largest possible $m$ occurs for the maximal parabolic subgroup
corresponding to the node with three neighbors in the Dynkin diagram
for $E_{8}$ when $m=6$. An illustrative example is when $G=\text{Sp}_{2n}$
and $P=N\rtimes M$ is the Siegel parabolic subgroup with Levi $M\simeq GL_{n}$.
In this case 
\[
c(s,\pi)=\frac{L(s,\pi)}{L(1+s,\pi)}\cdot\frac{L(2s,\pi,\wedge^{2})}{L(1+2s,\pi,\wedge^{2})},
\]
the Eisenstein series $E^{P}(g,\varphi,s)$ converges for $\Re(s)>\frac{n+1}{2}$,
and has a pole at $s=\frac{1}{2}$ if $L(s,\pi,\wedge^{2})$ has a
pole at $s=1$ \emph{and }$L(1/2,\pi)\neq0$ (see \cite{Lapid_Rallis_Annals}
for a further discussion of this example).

The case $G=G_{2}$ provided the extremely striking example of symmetric
cube $L$-functions attached to modular forms on the upper half-plane,
obtained \emph{without} recourse to Fourier coefficients. This computation
was a turning point in the theory of automorphic forms. The $L$-functions
appearing in these formulas for $c(s,\pi)$ have been thoroughly investigated
by Shahidi \cite{Shahidi_G2}\cite{Shahidi-Auto-L-func-book}. 

\subsection{Main theorem}

Let $P=N\rtimes M$ be a maximal parabolic subgroup of $G$. Take
$\varphi=1$ and let 
\[
\varphi_{s}(g)=\delta_{P}^{\frac{1}{2}}\cdot\varphi\cdot e^{\langle s\varpi,H_{P}(g)\rangle}=\delta_{P}^{\frac{1}{2}}\cdot e^{\langle s\varpi,H_{P}(g)\rangle}
\]
and 
\[
E^{P}(s,g):=\sum_{\gamma\in P(F)\backslash G(F)}\varphi_{s}(\gamma g)
\]
This is the simplest example of a \emph{non-cuspidal} Eisenstein series,
since $\varphi=1$ is \emph{not }a cuspform on $M(\mathbb{A})$, called
the \emph{unramified degenerate Eisenstein series attached to $P$.
}The map $s\mapsto E(s,\bullet)$ initially convergent for $\text{Re}(s)\gg0$
has a meromorphic continuation to $\mathbb{C}$ as a vector-valued
function with values in the space of smooth functions of uniform moderate
growth on $G(F)\backslash G(\mathbb{A})$ (see Bernstein-Lapid \cite{Lapid_Bernstein}).

In this paper, we obtain a polynomial $p\in\mathbb{C}[s]$ given in
terms of the structure of $P$, whose zeros capture the locations
and the orders of the poles of $E^{P}(s,g)$ in the region $\text{Re}(s)\ge0$.
Let $V_{r}$ be the $(r+1)$-dimensional irreducible representation
of $\mathfrak{sl}_{2}(\mathbb{C})$ given by the $r$-th symmetric
power of the standard representation. The notion of \emph{principal}
$\mathfrak{sl}_{2}\mathbb{C}$ for the Lie algebra of a split \emph{reductive}
group (with a fixed pinning) is defined in \ref{subsec:PrincipalSL2}
following Gross \cite{Gross_principal_SL2}. 
\begin{thm}
\label{thm:intro}Let $G$ be a split semi-simple linear algebraic
group over a number field and $P=N\rtimes M$ be the standard Levi
decomposition of a standard maximal parabolic subgroup $P$. Let 
\[
^{L}\mathfrak{n}=r_{1}\oplus r_{2}\oplus\cdots\oplus r_{m}
\]
where $r_{1},\dots,r_{m}$ are the irreducible constituents of the
adjoint representation of $^{L}M$ on $^{L}\mathfrak{n}$ as described
in \ref{subsec:LanglandsComp}. Let 
\[
r_{j}\simeq\bigoplus_{\ell\ge0}V_{\ell}^{m_{\ell}(j)},\quad V_{k}=\text{sym}^{k}(\text{std})
\]
be the decomposition of $r_{j}$ into irreducible constituents under
the action of the \textbf{principal} $\mathfrak{sl}_{2}\mathbb{C}\subset{}^{L}\mathfrak{m}$.
Let 
\[
p(s)=\prod_{j=1}^{m}\prod_{\ell\ge0}(js-1-\ell/2)^{m_{\ell}(j)}\ \in\mathbb{C}[s]
\]

In the region $\mathsf{Re}(s)\ge0$, $p(s)\cdot E^{P}(s,g)$ is holomorphic
and is a non-zero function on $G(\mathbb{A})$ for $\text{Re}(s)>0$. 

In other words, the order of the zeros of p(s) is the same as the
order of the poles of $E^{P}(s,g)$ in the region $\text{Re}(s)\ge0$. 
\end{thm}

We give a simple example to illustrate the theorem. See Fulton-Harris
\cite{Fult_Har}, chapter 11, for a method to decompose a representation
of $\mathfrak{sl}_{2}\mathbb{C}$ (abstractly) into irreducible constituents. 
\begin{example}
Let $G=PGL_{n+1}$ ($n\ge2$) and $P=N\rtimes M$ be the maximal parabolic
subgroup corresponding to the ordered partition $(a+1,b+1)$ of $n+1$
with $a+b=n-1$ so that the derived group of the Levi subgroup $M$
is of type $A_{a}\times A_{b}$ (Dynkin diagram notation). The dual
group $^{L}G=\text{SL}_{n+1}(\mathbb{C})$. The principal $\mathfrak{sl}_{2}\mathbb{C}\subset{}^{L}\mathfrak{m}$
is given by the $\mathfrak{sl}_{2}$-triple $\{H,X,Y\}$ where the
neutral element is 
\[
H=\text{diag}(a,a-2,\dots,-a,b,b-2,\dots,-b)
\]
We may identify $^{L}\mathfrak{n}=r_{1}$ with $(a+1)\times(b+1)$-matrices.
We write $H$-eigenvalues of the corresponding coroot vectors in the
matrix
\[
\begin{bmatrix}a-b & (a-b)-2 & \cdots & (a+b)-2 & (a+b)\\
(a-b)-2 &  &  &  & (a+b)-2\\
\vdots &  & \iddots &  & \vdots\\
2-(a+b) &  &  &  & (b-a)+2\\
-(a+b) & 2-(a+b) & \cdots & (b-a)-2 & b-a
\end{bmatrix}
\]
Then
\[
r_{1}\simeq\bigoplus_{k=a+b-2\min\{a,b\}}^{k=a+b}V_{k},\quad k\text{ increments of }2
\]
Note that $a+b=n-1$. The Eisenstein series $E^{P}(s,g)$ converges
for $\text{Re}(s)>\frac{n+1}{2}$ and has simple poles at 
\[
\frac{n+1}{2},\frac{n+1}{2}-1,\dots,\frac{n+1}{2}-\min\{a,b\}
\]
in the region $\text{Re}(s)>0$. 
\end{example}

\begin{rem*}
This was first shown by Hanzer and Muic \cite{Hanzer_Muic} by a detailed
study of $c_{B}E_{s}^{P}(g)$ and observing cancellation among a sum
of intertwining operators. See also the work of Koecher \cite{Koecher}. 
\end{rem*}

\subsection{Applications of the poles of degenerate Eisenstein series}

Degenerate Eisenstein series (including the ramified ones) and their
poles are of interest for several applications. We mention some of
them to indicate some of the previous work on the poles of unramified
degenerate Eisenstein series. 

\subsubsection{Siegel-Weil formula}

The classical Siegel-Weil formula \cite{Weil_Acta64} identifies the
integral of a certain theta series as the special value of an Eisenstein
series. The automorphic forms appearing in the Laurent expansion at
these poles of degenerate Eisenstein series play a central role in
the regularized versions of the Siegel-Weil formula (see \cite{Kudla_Rallis_SiegelWeil},
\cite{Gan_secondTerm}). Kudla and Rallis determined the poles of
Siegel parabolic degenerate Eisenstein series. For an important recent
work, see Halawi and Segal \cite{halawi2023poles}. 

\subsubsection{Integral representation for $L$-functions}

Degenerate Eisenstein series are used to obtain integral representations
of some automorphic $L$-functions. The information about the poles
can be used to obtain results about the poles of these automorphic
$L$-functions (see \cite{Cogdell_RankinSelberg}, \cite{Kudla_Rallis}). 

\subsubsection{Arthur conjecture}

Let $G$ be a split simple adjoint group $G$ over a number field.
The maximal parabolic unramified degenerate Eisenstein series are
functions of one complex variable and the leading term of the Laurent
expansion at \emph{some }of these poles is square-integrable. These
are used to obtain unitary representations of the adele group $G(\mathbb{A})$
and the local constituents of these representations are unitary. This
method was used by Miller \cite{Miller_Annals} to verify Arthur's
conjecture that the spherical constituents of principal series representations
at certain points of reducibility are unitary. 

\subsubsection{Spectral decomposition}

Poles of unramified degenerate Eisenstein series play a central role
in understanding Langlands' work \cite{Langlands-Spectral} on the
spectral decomposition of automorphic forms. In this work, Langlands
notes that ``A number of unexpected and unwanted complications must
be taken into account...'' One such complication is the \emph{cancellation
of residues during the contour deformation}. In particular, the case
of $G_{2}$ was first obtained by Langlands in appendix III of \cite{Langlands-Spectral}.
A first step in understanding this difficult work is to determine
the poles of the unramified degenerate Eisenstein series, which we
do in this paper. For some recent work on spectral decomposition,
see \cite{Opdam_and_Others}\cite{Kazhdan_Okunkov}. 

\subsection{Outline of the paper}

We introduce the unramified degenerate Eisenstein series in section
\ref{sec:Degenerate-Eisenstein-Series}, after recalling some relevant
standard notation due to Arthur. 

We need Langlands' theorem that unramified degenerate Eisenstein series
occur as ``residues'' of the unramified Borel Eisenstein series.
In section \ref{sec:min-para}, we recall this result and the properties
of the Borel Eisenstein series we need. 

Section \ref{sec:Poles} is the heart of the paper. The line of argument
is made clear in the case of $\text{SL}_{3}$ at the beginning of
this section \ref{subsec:SL3}. The argument uses the zeros and the
poles of the Borel Eisenstein series and the residue formula of Langlands
to determine the poles of unramified degenerate Eisenstein series
in the positive half-plane $\text{Re}(s)\ge0$. 

In the sections that follow we do explicit computations for several
classical and exceptional groups. 

\subsection*{Acknowledgment}

I wish to thank Professors Jean-Loup Waldspurger and Erez Lapid for
pointing out some errors in a previous draft of this manuscript. I
wish to express my gratitude to Professor Bill Casselman for his support
during this project. His notes online, and a draft note on $G_{2}$,
were very helpful to me. I thank Professor Hugo Chapdelaine whose
careful critique considerably improved the exposition.

I owe an enormous debt of gratitude to my thesis advisor Professor
Paul B. Garrett who treated me as his son. Almost everything I know
about number theory, I learned it from him.

\section{\label{sec:Degenerate-Eisenstein-Series}Unramified degenerate Eisenstein
series}

In this section we recall the standard notation due to Arthur \cite{Arthur-trace}.
Even though we only deal with maximal parabolic subgroups, it is clarifying
to introduce the algebraic preliminaries for the non-maximal cases. 

Let $F$ be a number field, $\mathbb{A}$ be the ring of adeles of
$F$, and $|\cdot|$ denote the adele norm on $\mathbb{A}$. Let $\mathbf{G}_{a}$
be the additive group over $F$ and $\mathbf{G}_{m}=GL_{1}$ be the
multiplicative group over $F$. The group $\mathbf{G}_{m}(\mathbb{A})$
of ideles for $F$ is denoted $\mathbb{J}$. 

\subsection{Homomorphism $H_{G}$}

Let $G$ be a connected linear algebraic group defined over $F$,
not necessarily reductive. Let 
\[
X_{F}(G)=\text{Hom}_{F}(G,\mathbf{G}_{m})
\]
denote the abelian group of $F$-rational characters of $G$ and let
\[
\mathfrak{a}_{G}=\text{Hom}_{\mathbb{Z}}\left(X_{k}(G),\mathbb{R}\right)\quad\text{and}\quad\mathfrak{a}_{G}^{*}=X_{k}(G)\otimes_{\mathbb{Z}}\mathbb{R}
\]
These are vector spaces over $\mathbb{R}$ and there is a natural
pairing $\langle\cdot,\cdot\rangle:\mathfrak{a}_{G}^{*}\times\mathfrak{a}_{G}\to\mathbb{R}$.
The homomorphism 
\[
H_{G}:G(\mathbb{A})\to\mathfrak{a}_{G}
\]
is 
\[
H_{G}(x):=\left[\chi\mapsto\log|\chi(x)|\right]
\]
Let 
\[
G(\mathbb{A})^{1}=\ker H_{G}\subset G(\mathbb{A})
\]
The function $H_{G}$ is trivial on $G(k)$. For $G=N\rtimes L$ a
Levi decomposition of $G$, $H_{G}$ is trivial on $N(\mathbb{A})$
and $L_{\text{der}}(\mathbb{A})$, where $L_{\text{der}}$ is the
derived group of $L$.

\subsection{Assumptions on $G$}

In discussing Eisenstein series, it is convenient to assume that $G$
is semi-simple so that $\mathfrak{a}_{G}=0$. However, the data on
the Levi subgroups of the parabolic subgroups must be defined for
reductive groups. We therefore begin with a connected reductive group
$G$ split over $F$. From subsection \ref{subsec:DecomLambda} onwards,
we put further restriction that $G$ is semi-simple. 

Let $G$ be a connected and reductive algebraic group. Let $Z_{G}$
be the center of $G$. Let $G_{\mathbb{Q}}$ be the restriction of
scalars of $F$ to $\mathbb{Q}$ and $A_{G}^{+}$ be the connected
component of the group of real points of the maximal $\mathbb{Q}$-split
torus in the center of $G_{\mathbb{Q}}$. Then 
\[
A_{G}^{+}\subset Z_{G}(F\otimes\mathbb{R})\subset Z_{G}(\mathbb{A})\subset G(\mathbb{A})
\]
The map 
\[
H_{G}:A_{G}^{+}\to\mathfrak{a}_{G}
\]
is an isomorphism. For a parabolic subgroup $P=N\rtimes M$ (Levi
decomposition) of $G$, we observe that $X_{F}(P)=X_{F}(M)$ and hence
\[
\mathfrak{a}_{P}=\mathfrak{a}_{M}.
\]

\subsection{Roots and Coroots }

Let $G$ be a connected reductive group split over $F$. For the rest
of the paper, fix a maximal split torus $T\subset G$ with Lie algebra
$\mathfrak{t}$. Let 
\[
X_{F}(T)=\text{Hom}(T,\mathbf{G}_{m})\quad\text{and}\quad X_{F}^{\vee}(T)=\text{Hom}(\mathbf{G}_{m},T)
\]
There is a natural pairing 
\[
X_{F}(T)\times X_{F}^{\vee}(T)\to\mathbb{Z};\quad(\chi,\eta)\mapsto\langle\chi,\eta\rangle
\]
defined by 
\[
\chi\circ\eta(x)=x^{\langle\chi,\eta\rangle},\quad\forall x\in\mathbf{G}_{m}
\]
Let $C_{G}(T)$ and $N_{G}(T)$ denote the centralizer and the normalizer
of $T$ in $G$. The Weyl group of the pair $(G,T)$ is 
\[
W=W(G,T)=N_{G}(T)/C_{G}(T)
\]

The adjoint action $\text{Ad}$ of $T$ on $\mathfrak{g}=\text{Lie}(G)$
is diagonalizable and 
\[
\mathfrak{g}=\mathfrak{t}\oplus\left(\bigoplus_{\alpha\in\Phi}\mathfrak{g}_{\alpha}\right)
\]
where $\Phi\subset X_{F}(T)$ is the finite set of \emph{roots} and
\[
\mathfrak{g}_{\alpha}=\left\{ x\in\mathfrak{g}:\text{Ad}(t)\cdot x=\alpha(t)x\text{ for all }t\in T\right\} 
\]
are $T$-eigenspaces. The only rational multiples of $\alpha$ in
$\Phi$ are $\pm\alpha$. 

For each root $\alpha\in\Phi$, the subtorus $T_{\alpha}:=\left(\ker\alpha\right)^{\circ}$
of $T$ has codimension $1$, where $\circ$ means the connected component
of the identity. Then $G_{\alpha}:=C_{G}(T_{\alpha})$ is connected
and 
\[
\text{Lie}(G_{\alpha})=\mathfrak{t}\oplus\mathfrak{g}_{\alpha}\oplus\mathfrak{g}_{-\alpha}
\]
The Weyl group of the pair $(G_{\alpha},T)$ has order $2$, and embeds
in $W(G,T)$. Let $w_{\alpha}\in W(G,T)$ be the non-identity element
of $W(G_{\alpha},T)$; then $w_{\alpha}$ acts on $X_{F}(T)$ as 
\[
w_{\alpha}(\chi)=\chi-\langle\chi,\alpha^{\vee}\rangle\alpha
\]
for a unique \emph{coroot} $\alpha^{\vee}\in X_{F}^{\vee}(T)$. Since
$w_{\alpha}^{2}=1$, we must have $\langle\alpha,\alpha^{\vee}\rangle=2$.
See \cite{MacDonald_AlgebraicGroups}\cite{Springer_Corvallis}, for
example. 

\subsection{Root Groups and a pinning}

For $\alpha\in\Phi$, there is a unique algebraic subgroup $U_{\alpha}\simeq\mathbf{G}_{a}$
of $G$, called the \emph{root group} corresponding to $\alpha$,
which is normalized by $T$ and on which the adjoint action of $T$
is through the character $\alpha$. The Lie algebra of $U_{\alpha}$
is $\mathfrak{g}_{\alpha}$. 

We fix a Borel subgroup $B$ defined over $F$ for the rest of this
paper. The \emph{root datum} attached to $T\subset B\subset G$ is
the quadruple
\[
(X_{F}(T),\Delta_{B},X_{F}^{\vee}(T),\Delta_{B}^{\vee})
\]
A \emph{pinning} (or splitting \cite{Springer_Corvallis}, pg. 9)
of $G$ for $T\subset B\subset G$ is a collection of isomorphisms
\[
\left\{ e_{\alpha}:\mathbf{G}_{a}\to U_{\alpha}\ |\ \alpha\in\Delta_{B}\right\} 
\]

\subsection{Parabolic subgroups}

Any $F$-subgroup $P$ of $G$ containing $B$ is a \emph{standard
}parabolic subgroup (relative to $B$). A standard parabolic subgroup
$P$ has a unique \emph{standard} Levi decomposition $P=N_{P}\rtimes M_{P}$
where $M_{P}$ contains $T$. 
\begin{rem*}
Since we only consider standard parabolic subgroups relative to $B$
and standard Levi decompositions, we drop the adjective \emph{standard}
and refer to \emph{the} Levi decomposition. 
\end{rem*}
Observe that $X_{F}(P)=X_{F}(M_{P})$ and $\mathfrak{a}_{P}=\mathfrak{a}_{M_{P}}$.
For notational convenience, we write $A_{P}:=A_{M_{P}}$ and note
that $A_{P}$ is \emph{not} central in $P$ to avoid any potential
confusion. The action of $A_{P}$ on $\mathfrak{n}_{P}:=\text{Lie}N_{P}$
is diagonalizable and 
\[
\mathfrak{n}_{P}=\bigoplus_{\beta\in\Phi_{P}}\mathfrak{n}_{\beta}
\]
where $\Phi_{P}$ is a finite subset of $\mathfrak{a}_{P}^{*}$ and
\[
\mathfrak{n}_{\beta}=\left\{ X\in\mathfrak{n}_{P}:\text{Ad}(a)X=\beta(a)X,\ \forall a\in A_{P}\right\} 
\]
Note that $\Phi_{B}$ is a set of positive roots in $\Phi$, and let
$\Delta_{B}$ be the corresponding set of simple roots. 

For each parabolic subgroup, let $\Delta_{B}^{P}\subset\Delta_{B}$
denote the subset of $\alpha\in\Delta_{B}$ appearing in the action
of $T$ in the unipotent radical of $B\cap M_{P}$. The correspondence
$P\to\Delta_{B}^{P}$ is a bijection between the set of standard parabolic
subgroups of $G$ and the set of subsets of $\Delta_{B}$. 

Let $\Delta_{P}$ be the set of linear forms on $\mathfrak{a}_{P}$
obtained by the restriction of the elements in $\Delta_{B}-\Delta_{B}^{P}$.
Then $\Delta_{P}$ is in bijection with $\Delta_{B}-\Delta_{B}^{P}$,
and any root in $\Phi_{P}$ can be written uniquely as a nonnegative
integral linear combination of elements in $\Delta_{P}$. 

\subsection{\label{subsec:DecomLambda}Decompositions of $\mathfrak{a}_{B}$
and $\mathfrak{a}_{B}^{*}$ }

For the rest of this section, assume that $G$ is semi-simple so that
$\mathfrak{a}_{G}=0$. Let $P\supset B$ be a parabolic subgroup.
The inclusions
\[
A_{P}\subset A_{B}\subset M_{B}\subset M_{P}
\]
give canonical decompositions
\[
\mathfrak{a}_{B}=\mathfrak{a}_{P}\oplus\mathfrak{a}_{B}^{P},\quad\mathfrak{a}_{B}^{*}=\mathfrak{a}_{P}^{*}\oplus\left(\mathfrak{a}_{B}^{P}\right)^{*}
\]
For any $\Lambda\in\mathfrak{a}_{B}^{*}$ and $H\in\mathfrak{a}_{B}$,
we write 
\begin{equation}
\Lambda=\Lambda_{P}+\Lambda_{B}^{P}\quad\text{where }\Lambda_{P}\in\mathfrak{a}_{P}^{*},\ \Lambda_{B}^{P}\in\left(\mathfrak{a}_{B}^{P}\right)^{*}\label{eq:Lambda}
\end{equation}
and
\begin{equation}
H=H_{P}+H_{B}^{P}\quad\text{where }H_{P}\in\mathfrak{a}_{P},\ H_{B}^{P}\in\mathfrak{a}_{B}^{P}\label{eq:H}
\end{equation}

\subsection{Relation between $H_{B}$ and $H_{P}$}

Let $K$ be the standard special maximal compact subgroup which provides
the Iwasawa decomposition $G(\mathbb{A})=P(\mathbb{A})K(\mathbb{A})$
for any standard parabolic subgroup $P$. We extend $H_{P}$ from
$P_{\mathbb{A}}$ to $G_{\mathbb{A}}$ by
\[
H_{P}:G_{\mathbb{A}}\to\mathfrak{a}_{P},\quad H_{P}(pk)=H_{P}(p)\quad p\in P_{\mathbb{A}},k\in K_{\mathbb{A}}
\]
The decomposition $\mathfrak{a}_{B}=\mathfrak{a}_{P}\oplus\mathfrak{a}_{B}^{P}$
gives
\[
H_{B}(x)=H_{P}(x)+H_{B}^{P}(x)\quad\forall x\in G_{\mathbb{A}}
\]
and the two definitions---one by extension of $H_{P}$ to $G_{\mathbb{A}}$
and the other as the projection of $H_{B}$ to $\mathfrak{a}_{P}$---are
the same object. 

\subsection{Basis for $\mathfrak{a}_{P}^{*}$ and $\mathfrak{a}_{P}$}

Let $\widehat{\Delta}_{B}=\left\{ \varpi_{\alpha}:\alpha\in\Delta_{B}\right\} $
be the set of \emph{fundamental weights}, defined by 
\[
\langle\varpi_{\alpha},\beta^{\vee}\rangle=\delta_{\alpha\beta}\quad(\text{Kronecker delta})\ \forall\alpha,\beta\in\Delta_{B}
\]
Then 
\[
\widehat{\Delta}_{P}=\left\{ \varpi_{\alpha}:\alpha\in\Delta_{B}-\Delta_{B}^{P}\right\} 
\]
is a basis for $\mathfrak{a}_{P}^{*}$. Let 
\[
\Delta_{P}^{\vee}=\left\{ \alpha^{\vee}:\alpha\in\Delta_{P}\right\} 
\]
be the dual basis of $\widehat{\Delta}_{P}$. For $\alpha\in\Delta_{P}$,
let $\beta\in\Delta_{B}-\Delta_{B}^{P}$ be the simple root whose
restriction to $\mathfrak{a}_{P}$ is $\alpha$. Then $\alpha^{\vee}$
is the canonical projection of $\beta^{\vee}\in\mathfrak{a}_{B}$
onto $\mathfrak{a}_{P}$.

\subsection{Unramified degenerate Eisenstein series}

Let $P$ be a parabolic subgroup and let
\[
\rho_{P}=\frac{1}{2}\sum_{\alpha\in\Phi_{P}}(\dim\mathfrak{n}_{\alpha})\alpha
\]
This defines $\rho_{P}$ and $\rho_{B}$ and the notation is consistent
with equation \ref{eq:Lambda} giving $\rho_{B}=\rho_{P}+\rho_{B}^{P}$.
Let
\[
E_{\Lambda}^{P}(g)=\sum_{\gamma\in P_{k}\backslash G_{k}}e^{\langle\rho_{P}+\Lambda,H_{P}(\gamma g)\rangle},\quad\Lambda\in\mathfrak{a}_{P}^{*}\otimes\mathbb{C}.
\]
It converges absolutely and uniformly on compact subsets of $G_{\mathbb{A}}$
when $\Lambda$ is in the tube $\rho_{P}+T_{P}$, where
\[
T_{P}:=\left\{ \Lambda\in\mathfrak{a}_{P}^{*}\otimes\mathbb{C}:\Re\langle\Lambda,\alpha^{\vee}\rangle>0,\ \forall\alpha\in\Delta_{P}\right\} 
\]
is the \emph{tube} over the \emph{positive cone} $C_{P}:=T_{P}\cap\mathfrak{a}_{P}^{*}$.
We call this the \emph{positive tube }$T_{P}$ to simplify terminology.
The function $E_{\Lambda}^{P}(g)$ has a meromorphic continuation
to $\mathfrak{a}_{P}^{*}\otimes\mathbb{C}$ (see \cite{Lapid_Bernstein}).
For $P\neq B,G$, we call $E_{\Lambda}^{P}(g)$ \emph{the unramified
degenerate Eisenstein series attached to $P$. }

\subsection{A remark on terminology }

We do \emph{not} consider the ramified cases in this paper. We drop
the word \emph{unramified} to lighten the presentation, and all further
references to degenerate Eisenstein series shall be taken to mean
unramified degenerate Eisenstein series.

\section{\label{sec:min-para}Minimal parabolic Eisenstein series}

Throughout this section, let $G$ be an $F$\emph{-split} \emph{semi-simple}
linear algebraic group over $F$ with $T\subset B\subset G$ for a
maximal $F$-split torus contained in a Borel subgroup $B$ defined
over $F$. 

In general, \emph{cuspidal }Eisenstein series are more tractable than
general Eisenstein series in their analytic behavior since their constant
terms---from which many analytic properties of Eisenstein series
follow---are easier to compute. The main result of the paper provides
evidence that the poles of \emph{non-cuspidal} Eisenstein series occurring
in the spectral decomposition can be understood from the \emph{zeros
}and \emph{poles} of cuspidal Eisenstein series. 

\subsection{Unramified Borel Eisenstein series}

Let $B=N\rtimes M$ be the Levi decomposition of $B$. We write 
\[
a(g)^{\Lambda}:=e^{\langle\Lambda,H_{B}(g)\rangle}
\]
to simplify the notation. For $P=B$, we call the Eisenstein series
$E_{\Lambda}^{B}(g)$ \emph{the} minimal parabolic (or \emph{the }Borel)
Eisenstein series: 
\[
E_{\Lambda}^{B}(g)=\sum_{\gamma\in B_{k}\backslash G_{k}}a(g)^{\rho_{B}+\Lambda}\quad\Lambda\in T_{B},\ g\in G_{\mathbb{A}}
\]
The function $E_{\Lambda}^{B}(g)$ converges absolutely in the positive
tube $T_{B}$ and has a meromorphic continuation to $\mathfrak{a}_{B}^{*}\otimes\mathbb{C}$.
For $w\in W$, it satisfies the functional equation
\[
E_{\Lambda}^{B}(g)=c_{w,\Lambda}E_{w\cdot\Lambda}^{B}(g)
\]
where 
\[
c_{w,\Lambda}=\prod_{\substack{\alpha\in\Phi_{B}:\\
w\cdot\alpha<0
}
}\frac{\xi(\langle\Lambda,\alpha^{\vee}\rangle)}{\xi(1+\langle\Lambda,\alpha^{\vee}\rangle)}
\]
and 
\[
\xi(s):=\ensuremath{\xi_{F}(s)}\text{ is the completed zeta function of }F
\]
Unlike $P\neq B$, when $P=B$, we get a \emph{cuspidal }Eisenstein
series in a vacuous, but meaningful, sense since the trivial character
on $k^{\times}\backslash\mathbb{J}_{k}^{1}$ is a cuspform for $GL_{1}$. 

\subsection{Poles of $E_{\Lambda}^{B}(g)$ \label{subsec:Poles}}

The constant term $c_{B}E_{\Lambda}^{B}$ of $E_{\Lambda}^{B}$ along
$B$ is
\[
c_{B}E_{\Lambda}^{B}(g):=\int_{N(F)\backslash N(\mathbb{A})}E_{\Lambda}^{B}(ng)dn
\]
It is a function on $N(\mathbb{A})M(F)\backslash G(\mathbb{A})$.
It was first computed by Gelfand et al. \cite{Gelfand_ICM} (page
82): 
\[
c_{B}E_{\Lambda}^{B}(g)=\sum_{w\in W}c_{w,\Lambda}a(g)^{\rho_{B}+w\cdot\Lambda}
\]

From the above formula, we see that the singularities of $c_{B}E_{\Lambda}^{B}(g)$
are hyperplanes of the form
\[
S(\alpha,c):=\left\{ \Lambda\in\mathfrak{a}_{B}^{*}\otimes\mathbb{C}:\langle\Lambda,\alpha^{\vee}\rangle=c\right\} 
\]
where $\alpha\in\Phi_{B}$ and $c\in\mathbb{C}$. Following Langlands,
we say that the singularity along $S(\alpha,c)$ is \emph{real} if
$c\in\mathbb{R}$. In the positive tube $T_{B}$, the singularities
of $E_{\Lambda}^{B}(g)$ and $c_{B}E_{\Lambda}^{B}$ are the same,
are real, and are given by 
\begin{equation}
S_{\gamma}:=S(\gamma,1)\quad\text{(}\gamma\in\Phi_{B}\text{)}\label{eq:S_alpha}
\end{equation}
These singularities are \emph{simple} in the sense that
\[
\Lambda\mapsto\prod_{\gamma\in\Phi_{B}}\left(\langle\Lambda,\gamma^{\vee}\rangle-1\right)E_{\Lambda}^{B}(g)
\]
extends to a holomorphic function on the positive tube $T_{B}$. 

\subsection{Zeros of $E_{\Lambda}^{B}(g)$}

\subsubsection{The $SL_{2}$ Eisenstein series}

As a prelude to the more general case, consider the $\text{SL}_{2}$
Eisenstein series $E_{s}(z)$ in the introduction. Its constant term
is
\[
\int_{0}^{1}E_{s}(x+iy)dx=y^{\frac{1}{2}+\frac{s}{2}}+c(s)y^{\frac{1}{2}-\frac{s}{2}},
\]
The constant term vanishes at $s_{0}$ as a function of $y$ only
if 
\[
c(s_{0})=-y^{s_{0}}\quad\text{for all }y>0
\]
This can happen only for $s_{0}=0$ and if 
\[
\lim_{s\to0}c(s)=-1.
\]
This is indeed true, since for any number field $F$, the corresponding
completed zeta function $\xi:=\xi_{F}$ satisfies
\[
\lim_{s\to0}\frac{\xi(s)}{\xi(s+1)}=\frac{\mathsf{res}_{s=0}\xi(s)}{\mathsf{res}_{s=1}\xi(s)}=-1.
\]
Further, the zero of $E_{s}(z)$ at $0$ is simple.

\subsubsection{The general case}

We say that $\Lambda\in\mathfrak{a}_{B}^{*}\otimes\mathbb{C}$ is
\emph{regular }if $\Lambda$ is not fixed by any $w\in W$. For regular
$\Lambda\in\mathfrak{a}_{B}^{*}\otimes\mathbb{C}$, the set 
\[
\left\{ a(g)^{w\cdot\Lambda}:w\in W\right\} 
\]
is a linearly independent set of functions on $G(\mathbb{A})$ and
\[
c_{B}E_{\Lambda}^{B}(g)=0\iff\sum_{w\in W}c_{w,\Lambda}a(g)^{w\cdot\Lambda}
\]
\[
\iff c_{w,\Lambda}=0\ \text{for all }w\in W
\]
Since $c_{1,\Lambda}=1$, we conclude that the Eisenstein series $E_{\Lambda}^{B}\neq0$
for regular $\Lambda\in\mathfrak{a}_{B}^{*}\otimes\mathbb{C}$. Let
\[
H_{\alpha}:=\left\{ \Lambda\in\mathfrak{a}_{B}^{*}\otimes\mathbb{C}:\langle\Lambda,\alpha^{\vee}\rangle=0\right\} ,\quad\alpha\in\Phi_{B}
\]
Note that every $\Lambda$ in the complement of $\cup_{\alpha\in\Phi_{B}}H_{\alpha}$
in $\mathfrak{a}_{B}^{*}\otimes\mathbb{C}$ is regular. The set of
regular elements is an open dense subset of $\mathfrak{a}_{B}^{*}\otimes\mathbb{C}$. 

The following result of Jacquet \cite{Jacquet_residual_GL_n} computes
all possible zeros of the Borel Eisenstein series $E_{\Lambda}^{B}$. 
\begin{prop}
The Eisenstein series $E_{\Lambda}^{B}(g)$ has a simple zero along
the hyperplanes $H_{\alpha}$ for $\alpha\in\Phi_{B}$. 
\end{prop}

\begin{proof}
(Jacquet) From the general theory of Eisenstein series, it is enough
to show that for generic $\Lambda\in H_{\alpha}$ 
\[
c_{B}E_{\Lambda}^{B}(g)=0
\]
for each $\alpha\in\Phi_{B}$ and that the zero along $H_{\alpha}$
is simple. We first prove this for a simple root $\alpha\in\Delta_{B}$
and deduce the case when $\alpha$ is not simple using the functional
equations of $E_{\Lambda}^{B}(g)$. 

\textbf{Step 1: }$\alpha$ simple. Let $w_{\alpha}\in W$ be the reflection
corresponding to a simple root $\alpha\in\Delta_{B}$. The group $W_{\alpha}=\{1,w_{\alpha}\}$
acts on $W$ on the right with orbits of the form $\{w,w\cdot w_{\alpha}\}$
for $w\in W$. Using 
\[
c_{ww_{\alpha},\Lambda}=c_{w,w_{\alpha}\Lambda}\cdot c_{w_{\alpha},\Lambda}\quad\text{(cocycle relation)},
\]
\[
w_{\alpha}\cdot\Lambda=\Lambda\quad\text{on }H_{\alpha},
\]
and 
\[
\lim_{\langle\Lambda,\alpha^{\vee}\rangle\to0}c_{w_{\alpha},\Lambda}=\lim_{\langle\Lambda,\alpha^{\vee}\rangle\to0}\frac{\xi(\langle\Lambda,\alpha^{\vee}\rangle)}{\xi(1+\langle\Lambda,\alpha^{\vee}\rangle)}=-1,
\]
we get 
\[
\lim_{\langle\Lambda,\alpha^{\vee}\rangle\to0}\left(c_{w,\Lambda}a(g)^{\rho_{B}+w\cdot\Lambda}+c_{ww_{\alpha},\Lambda}a(g)^{\rho_{B}+ww_{\alpha}\cdot\Lambda}\right)
\]
\[
=\lim_{\langle\Lambda,\alpha^{\vee}\rangle\to0}\left(c_{w,\Lambda}a(g)^{\rho_{B}+w\cdot\Lambda}+c_{w,w_{\alpha}\cdot\Lambda}c_{w_{\alpha},\Lambda}a(g)^{\rho_{B}+ww_{\alpha}\cdot\Lambda}\right)=0
\]
away from the singularities of $c_{w,\Lambda}$. Breaking up the following
sum over $W$ by the orbits of the $W_{\alpha}$ action, 
\[
\lim_{\langle\Lambda,\alpha^{\vee}\rangle\to0}c_{B}E_{\Lambda}^{B}(g)=\lim_{\langle\Lambda,\alpha^{\vee}\rangle\to0}\sum_{w\in W}c_{w,\Lambda}a(g)^{\rho_{B}+w\cdot\Lambda}
\]
\[
=\sum_{\dot{w}\in W/W_{\alpha}}\lim_{\langle\Lambda,\alpha^{\vee}\rangle\to0}\left(c_{\dot{w},\Lambda}a(g)^{\rho_{B}+\dot{w}\cdot\Lambda}+c_{\dot{w}w_{\alpha},\Lambda}a(g)^{\rho_{B}+\dot{w}w_{\alpha}\cdot\Lambda}\right)=0
\]
For generic $\Lambda\in H_{\alpha}$, the set $\{a^{\dot{w}\Lambda}:\dot{w}\in W/W_{\alpha}\}$
is a linearly independent set of functions on $G(\mathbb{A})$. The
simplicity of the zero along $H_{\alpha}$ follows from the observation
that
\[
\lim_{\langle\Lambda,\alpha^{\vee}\rangle\to0}\frac{1}{\langle\Lambda,\alpha^{\vee}\rangle}\left(a(g)^{\rho_{B}+\dot{w}\cdot\Lambda}+c_{w_{\alpha},\Lambda}a(g)^{\rho_{B}+w_{\alpha}\cdot\Lambda}\right)\neq0
\]
for generic $\Lambda\in H_{\alpha}$ and that the term corresponding
to $\dot{w}=1$ is non-zero. The non-vanishing above is similar to
the fact that the $\text{SL}_{2}$ Eisenstein series $E_{s}(z)$ has
a simple zero at $s=0$. 

\textbf{Step 2. }$\alpha$ non-simple. Now we prove the vanishing
for a general positive root. Given a positive root $\beta\in\Phi_{B}$,
there exists $w\in W$ and a simple root $\alpha\in\Delta_{B}$ such
that $w\cdot\beta=\alpha$. We have the functional equation
\[
E_{\Lambda}^{B}(g)=c_{w,\Lambda}E_{w\cdot\Lambda}^{B}(g),
\]
Since $w\cdot\beta=\alpha>0$, the formula 
\[
c_{w,\Lambda}=\prod_{\substack{\gamma\in\Phi_{B}:\\
w\cdot\gamma<0
}
}\frac{\xi(\langle\Lambda,\gamma^{\vee}\rangle)}{\xi(1+\langle\Lambda,\gamma^{\vee}\rangle)}
\]
shows that $H_{\beta}$ is not a singular hyperplane of $c_{w,\Lambda}$.
Since 
\[
\lim_{s\to0}\frac{\xi(s)}{\xi(s+1)}\neq0,
\]
it follows that $c_{w,\Lambda}$ does not vanish along $H_{\beta}$.
Using $w\cdot H_{\beta}=H_{\alpha}$, we conclude that $c_{B}E_{\Lambda}^{B}(g)$
and $E_{\Lambda}^{B}(g)$ have simple zeros along $H_{\beta}$. 
\end{proof}

\subsection{$E_{\Lambda_{P}}^{P}$ as a residue of $E_{\Lambda}^{B}$}

The result of this subsection is the well-known theorem of Langlands
that the non-cuspidal Eisenstein series occurring in the spectral
decomposition of automorphic forms are ``residues'' of cuspidal
Eisenstein series (see Moeglin \cite{Moeglin_ICM}). The general notion
of residue required to prove this result is discussed in chapter 7
of Langlands \cite{Langlands-Spectral} and section V.1 of Moeglin-Waldspurger
\cite{MoeglinWaldspurger-book}. 

The case we need is the simplest and occurs without any of the complications
of the general case (see Langlands \cite{Langlands_Boulder,Langlands_Volume}).
For a parabolic subgroup $P$ of $G$, the set
\[
S_{P}:=\bigcap_{\alpha\in\Delta_{B}^{P}}S_{\alpha}\quad\text{(see }\text{\ref{eq:S_alpha} for the definition of }S_{\alpha})
\]
is an affine subspace of $\mathfrak{a}_{B}^{*}\otimes\mathbb{C}$.
The function 
\[
\prod_{\alpha\in\Delta_{B}^{P}}\left(\langle\Lambda,\alpha^{\vee}\rangle-1\right)\cdot E_{\Lambda}^{B}
\]
extends to a meromorphic function on $S_{P}$. The \emph{residue}
of $E_{\bullet}^{B}$ along $S_{P}$ is 
\[
\left(\mathsf{Res}_{S_{P}}E_{\bullet}^{B}\right):=\left(\prod_{\alpha\in\Delta_{B}^{P}}\left(\langle\Lambda,\alpha^{\vee}\rangle-1\right)\cdot E_{\Lambda}^{B}\right)\bigg|_{S_{P}}
\]
It is a meromorphic function on $S_{P}=\rho_{B}^{P}+\mathfrak{a}_{P}^{*}\otimes\mathbb{C}$. 

The following well-known result shows how the degenerate Eisenstein
series occur as residues of the minimal parabolic Eisenstein series. 
\begin{prop}
(Langlands) \label{prop: residue}For the decomposition \ref{subsec:DecomLambda}
\[
\Lambda=\Lambda_{P}+\Lambda_{B}^{P},
\]
we have 
\[
\left(\mathsf{Res}_{S_{P}}E_{\bullet}^{B}\right)(\rho_{B}^{P}+\Lambda_{P})=c\cdot E_{\Lambda_{P}}^{P}
\]
for some $c\neq0$. 
\end{prop}

\section{\label{sec:Poles}Poles of maximal parabolic degenerate Eisenstein
series}

Let $T\subset B\subset G$ be as before for a connected semi-simple
algebraic group $G$. Let $P$ be a maximal $F$-parabolic subgroup
$P\supset B$. We have $\Delta-\Delta_{B}^{P}=\{\beta\}$ for some
simple root $\beta\in\Delta_{B}$. Let $\varpi\in\mathfrak{a}_{B}^{*}$
be the fundamental weight dual vector to the coroot $\beta^{\vee}$. 

The vector $\varpi\in\mathfrak{a}_{P}^{*}$ and we parametrize $\mathfrak{a}_{P}^{*}\otimes\mathbb{C}$
by $s\varpi$. Let 
\[
E_{s\varpi}^{P}(g)=\sum_{\gamma\in P(F)\backslash G(F)}e^{\langle\rho_{P}+s\varpi,H_{P}(g)\rangle}
\]
It converges for $\text{Re}(s)\gg0$ and has a meromorphic continuation
to $\mathbb{C}$. 

In this section, we show that the poles of $E_{s\varpi}^{P}(g)$ in
the region $\text{Re}(s)\ge0$ are determined by the zeros of a polynomial
$p\in\mathbb{C}[s]$ obtained using the structure of $P$. Before
we treat the general case, we discuss the simplest example of $SL_{3}$
which highlights the issues we must address.

\subsection{\label{subsec:SL3}The $SL_{3}$ example}

The poles of $E_{\Lambda}^{B}$ \emph{outside the positive tube} can
come from the \emph{critical zeros} of $\xi:=\xi_{F}$ the completed
zeta function of the number field $F$, as can be seen from the formula
\[
c_{B}E_{\Lambda}^{B}(g)=\sum_{w\in W}c_{w,\Lambda}a(g)^{\rho_{B}+w\cdot\Lambda},\quad c_{w,\Lambda}=\prod_{\substack{\alpha>0:\\
w\cdot\alpha<0
}
}\frac{\xi(\langle\Lambda,\alpha^{\vee}\rangle)}{\xi(1+\langle\Lambda,\alpha^{\vee}\rangle)}
\]
Casselman \cite{Casselman_Eis_conj} brought attention to a curious
phenomena where the poles from the critical zeros do not contribute
to the poles meeting the positive tube of the degenerate Eisenstein
series $E_{\Lambda_{P}}^{P}$, even though the point $\rho_{B}^{P}\in S_{P}\subset\mathfrak{a}_{B}^{*}\otimes\mathbb{C}$
is outside the positive cone. Following Casselman, we illustrate this
phenomenon for $SL_{3}$. 

Let $G=SL_{3}$. Let $\alpha,\beta\in\Delta_{B}$ be the simple roots.
Let 
\[
s_{\alpha}=\langle\Lambda,\alpha^{\vee}\rangle\quad\text{and}\quad s_{\beta}=\langle\Lambda,\beta^{\vee}\rangle
\]

\begin{example}
For $w\in W$ the \emph{longest} Weyl element, 
\[
c_{w,\Lambda}=\frac{\xi(s_{\alpha})\xi(s_{\beta})\xi(s_{\alpha}+s_{\beta})}{\xi(1+s_{\alpha})\xi(1+s_{\beta})\xi(1+s_{\alpha}+s_{\beta})}.
\]
For $\Delta_{B}^{P}=\{\alpha\}$, we have $\rho_{B}^{P}=\frac{1}{2}\alpha$,
$\langle\rho_{B}^{P},\beta^{\vee}\rangle=-\frac{1}{2}$, $\langle\varpi,\alpha^{\vee}\rangle=0$,
and $\langle\varpi,\beta^{\vee}\rangle=1$. 

For $\Lambda=\rho_{B}^{P}+s\varpi$, we have 
\[
s_{\beta}=\langle\Lambda,\beta^{\vee}\rangle=\langle\rho_{B}^{P}+s\varpi,\beta^{\vee}\rangle=-\frac{1}{2}+s
\]
and the term $\xi(1+s_{\beta})$ in the denominator of $c_{w,\Lambda}$
could contribute poles from the \emph{critical zeros} of $\xi(s)$
to $E_{s\varpi}^{P}(g)$ in the region $\text{Re}(s)\ge0$.

However, 
\[
\mathsf{Res}_{S_{P}}c_{w,\Lambda}=\frac{\mathsf{Res}_{z=1}\xi(z)}{\xi(2)}\cdot\frac{\xi(s_{\beta})\cancel{\xi(1+s_{\beta})}}{\cancel{\xi(1+s_{\beta})}\xi(2+s_{\beta})}
\]
and the troublesome term is cancelled. 
\end{example}

The reader should consult Casselman's \cite{Casselman_Eis_conj} account
for further examples of cancellations of this sort. In fact, this
paper is the inspiration to all the ideas in this paper. This cancellation
is \emph{not }sufficient to determine the poles of degenerate Eisenstein
series, not even for $SL_{3}$. The following example illustrates
what is going on. 
\begin{example}
The poles and zeros of $E_{\Lambda}^{B}(g)$ relevant for determining
the poles of degenerate Eisenstein series in our region of interest
are captured by the meromorphic function 
\[
F(\Lambda):=\frac{s_{\alpha}}{s_{\alpha}-1}\cdot\frac{s_{\beta}}{s_{\beta}-1}\cdot\frac{s_{\alpha}+s_{\beta}}{s_{\alpha}+s_{\beta}-1}
\]
Note that when $s_{\alpha}=1$, we can have a pole at 
\[
(s_{\alpha},s_{\beta})=(1,1)\text{ and }(1,0).
\]
However, 
\[
\mathsf{Res}_{S_{\alpha}}F(\Lambda)=(s_{\alpha}-1)\frac{s_{\alpha}}{s_{\alpha}-1}\cdot\frac{s_{\beta}}{s_{\beta}-1}\cdot\frac{s_{\alpha}+s_{\beta}}{s_{\alpha}+s_{\beta}-1}\bigg|_{S_{\alpha}}
\]
\[
=\frac{\cancel{s_{\beta}}}{s_{\beta}-1}\cdot\frac{1+s_{\beta}}{\cancel{s_{\beta}}}\bigg|_{S_{\alpha}}=\frac{1+s_{\beta}}{s_{\beta}-1}\bigg|_{S_{\alpha}}
\]
The cancellation above explains why the degenerate Eisenstein series
obtained by taking residue along $S_{\alpha}$ ($\alpha$ a simple
root) is \emph{holomorphic} at the point $S_{\alpha}\cap S_{\gamma}$,
where $\gamma$ is the non-simple positive root. 
\end{example}

This simple observation is sufficient to obtain both the locations
and the order of the poles. To explicate the above cancellations,
we need the principal $\mathfrak{sl}_{2}\mathbb{C}$ subalgebras of
$^{L}\mathfrak{g}$ discussed below. In subsections \ref{subsec:PrincipalSL2}
and \ref{subsec:ImpLemma}, we drop the hypothesis that $G$ be semi-simple. 

\subsection{\label{subsec:PrincipalSL2}Principal homomorphism $\text{SL}_{2}\mathbb{C}\to{}^{L}G$ }

For a connected reductive group $G$ split over $k$ with Lie algebra
$\mathfrak{g}$, its root datum consists of a maximal torus $T\subset B$
and the quadruple
\[
(X_{k}(T),\Delta_{B},X_{k}^{\vee}(T),\Delta_{B}^{\vee})
\]
A \emph{pinning/splitting} of $G$ for $T\subset B\subset G$ is a
collection of isomorphisms 
\[
\left\{ e_{\alpha}:\mathbf{G}_{a}\to U_{\alpha}\ |\ \alpha\in\Delta_{B}\right\} 
\]
The construction of the \emph{dual group }$^{L}G$ of $G$ gives a
$k$-split torus $^{L}T$, a Borel $^{L}B\supset{}^{L}T$, and the
root datum 
\[
\left(X_{k}({}^{L}T)=X_{k}^{\vee}(T),\ \Delta_{^{L}B}\simeq\Delta_{B}^{\vee},\ X_{k}^{\vee}({}^{L}T)=X_{k}(T),\ \Delta_{^{L}B}^{\vee}\simeq\Delta_{B}\right)
\]
and root vectors 
\[
\left\{ e_{\alpha^{\vee}}:\mathbf{G}_{a}\to U_{\alpha^{\vee}}\ |\ \alpha\in\Delta_{B}\right\} 
\]
We have an identification between the positive roots of $^{L}B$ in
$\text{Hom}({}^{L}T,\mathbf{G}_{m})$ and the positive coroots for
$B$ in $\text{Hom}(\mathbf{G}_{m},T)$ (see Springer \cite{Springer_Corvallis}). 

Let $\Delta:=\Delta_{B}$, $\Phi^{+}:=\Phi_{B}$, and $^{L}\mathfrak{g}:=\text{Lie}(^{L}G)$.
Let
\[
x_{\alpha^{\vee}}:=\text{Lie}(e_{\alpha^{\vee}})(1)\quad\text{in }{}^{L}\mathfrak{g}_{\alpha^{\vee}}=\text{Lie}(U_{\alpha^{\vee}})
\]
and
\[
X:=\sum_{\alpha\in\Delta}x_{\alpha^{\vee}}
\]
Then $X$ is a principal nilpotent element in $^{L}\mathfrak{g}=\text{Lie}(^{L}G)$
(sometimes also referred to as regular nilpotent element). For each
$\alpha^{\vee}\in\Phi^{\vee}$, let $h_{\alpha^{\vee}}\in{}^{L}\mathfrak{g}$
be the vector determined by the coroot $\alpha^{\vee}:\mathbb{G}_{m}\to T$,
and let 
\[
H=\sum_{\gamma\in\Phi^{+}}h_{\gamma^{\vee}}=\sum_{\alpha\in\Delta}c_{\alpha}h_{\alpha^{\vee}}
\]
The coefficients $c_{\alpha}$ are \emph{positive} integers. Finally,
for each $\alpha\in\Delta$, let $y_{\alpha^{\vee}}$ be the unique
basis of $\text{Lie}(U_{-\alpha^{\vee}})$ such that 
\[
[x_{\alpha^{\vee}},y_{\alpha^{\vee}}]=h_{\alpha^{\vee}}
\]
Let 
\[
Y=\sum_{\alpha\in\Delta}c_{\alpha}y_{\alpha^{\vee}}
\]
A simple calculation shows that $\{H,X,Y\}$ is a standard $\mathfrak{sl}_{2}$-triple:
\[
[H,X]=2X,\quad[H,Y]=-2Y,\quad[X,Y]=H
\]
There is a homomorphism 
\[
\phi:\mathfrak{sl}_{2}(\mathbb{C})\to{}^{L}\mathfrak{g}
\]
given by
\[
\left(\begin{matrix}1 & 0\\
0 & -1
\end{matrix}\right)\mapsto H,\quad\left(\begin{matrix}0 & 1\\
0 & 0
\end{matrix}\right)\mapsto X,\quad\left(\begin{matrix}0 & 0\\
1 & 0
\end{matrix}\right)\mapsto Y
\]
and since $SL_{2}(\mathbb{C})$ is simply connected, we have a homomorphism
of reductive groups $\varphi:SL_{2}(\mathbb{C})\to{}^{L}G$. We refer
to it as \emph{the principal homomorphism} $SL_{2}\to{}^{L}G$. The
co-character $\mathbb{G}_{m}\to{}^{L}T$ given by the restriction
of $\varphi$ to the maximal torus 
\[
\mathbf{G}_{m}\simeq\left\{ \left(\begin{matrix}t & 0\\
0 & t^{-1}
\end{matrix}\right)\subset SL_{2}:t\in\mathbf{G}_{m}\right\} 
\]
is $2\rho_{B}$ in $\text{Hom}(\mathbb{G}_{m},{}^{L}T)=\text{Hom}(T,\mathbb{G}_{m})$
(see Gross \cite{Gross_principal_SL2}). 

\subsection{\label{subsec:ImpLemma}The action of the principal $\mathfrak{sl}_{2}\mathbb{C}$
in $\mathfrak{m}$ on $\mathfrak{n}$}

Let $G$ be a split connected reductive group and $P=N\rtimes M$
be the standard Levi decomposition of a parabolic subgroup $P$. To
explicate in general the cancellations illustrated in \ref{subsec:SL3}
for $SL_{3}$, we need to study the action of the principal $SL_{2}$
in $M$ on $\mathfrak{n}$. We first discuss an example. 
\begin{example}
\label{exa:GL4}Let $G=GL_{4}$ with standard Levi $M=GL_{2}\times GL_{2}$
of the standard parabolic $P$. The space $\mathfrak{n}$ is the space
of $2\times2$ matrices (abelian Lie algebra). Let $\alpha_{1},\alpha_{2},\alpha_{3}$
be the standard numbering for the simple roots of the standard maximal
torus of $G$. 

For $i,j\in\{1,2\}$, denote $E_{ij}$ the matrix whose $(i',j')$-coefficient
is $1$ if $(i',j')=(i,j)$ and $0$ if $(i',j')\neq(i,j)$. The lines
$\mathbb{C}E_{ij}$ are eigenspaces for the action of the standard
torus of $G$, associated to the root $\alpha_{2}$ for $(i,j)=(2,1)$,
$\alpha_{1}+\alpha_{2}$ for $(i,j)=(1,1)$, $\alpha_{2}+\alpha_{3}$
for $(i,j)=(2,2)$, and $\alpha_{1}+\alpha_{2}+\alpha_{3}$ for $(i,j)=(1,2).$

The principal $SL_{2}$ is the diagonal embedding of $SL_{2}\to GL_{2}\times GL_{2}$
and $SL_{2}$ acts by conjugation on $\mathfrak{n}$. This representation
decomposes as 
\[
\mathfrak{n}\simeq V_{0}\oplus V_{2},\quad V_{k}\simeq\text{Sym}^{k}(\text{std})
\]
where 
\[
V_{0}=\mathbb{C}(E_{11}+E_{22})\quad\text{and}\quad V_{2}=\mathbb{C}E_{21}\oplus\mathbb{C}(E_{11}-E_{22})\oplus\mathbb{C}E_{12}
\]
Note that $V_{0}$ is \emph{not} a root space. 
\end{example}

The following simple observation is critical in the proof of proposition
\ref{prop:franke}.
\begin{lem}
\label{lem:X-action}Let $P=N\rtimes M$ be the standard Levi decomposition
of a standard proper parabolic subgroup $P$ of $G$. Let $X$ be
the principal nilpotent element in $\mathfrak{m}$ and $\alpha\in\Phi_{B}$
be a positive root with $\mathfrak{g}_{\alpha}\subset\mathfrak{n}$.
Then $X\cdot\mathfrak{g}_{\alpha}\subset\bigoplus_{\theta\in\Delta_{B}^{P}}\mathfrak{g}_{\alpha+\theta}$.
In particular, if $X\cdot\mathfrak{g}_{\alpha}\neq0$ then there exists
$\theta\in\Delta_{B}^{P}$ such that $\alpha+\theta\in\Phi_{B}$. 
\end{lem}

\begin{proof}
Let $\Theta=\Delta_{B}^{P}$ for notational convenience. We have $X=\sum_{\theta\in\Theta}x_{\theta}$
and 
\[
X\cdot x_{\alpha}=\sum_{\theta\in\Theta}[x_{\theta},x_{\alpha}]=\sum_{\theta\in\Theta}a_{\theta}x_{\alpha+\theta}\subset\bigoplus_{\theta\in\Theta}\mathfrak{g}_{\alpha+\theta}
\]
for some constants $a_{\theta}$ for $\theta\in\Theta$. 
\end{proof}

The action of the principal $SL_{2}$ in $M$ commutes with the central
torus $A_{P}\subset M$. For 
\[
\mathfrak{n}=\bigoplus_{\alpha\in\Phi_{P}}\mathfrak{n}_{\alpha}
\]
we get an action of the principal $SL_{2}$ in $M$ on each $\mathfrak{n}_{\alpha}$
($\alpha\in\Phi_{P}$). This action plays a central role in the arguments
below.

\subsection{No contribution from the critical zeros of $\xi$}

We first make a comment about the poles of residues. Let $f$ be a
meromorphic function on a complex vector space whose singularities
are a locally finite collection $\{L_{\mu}:\mu\in I\}$ ($I$ is an
indexing set) of affine hyperplanes. Assume that the singularity along
a hyperplane $L$ is \emph{simple} so that the notion of residue is
straightforward. The residue of $f$ along $L$ is a meromorphic function
on $L$ and can only have singularities along $L\cap L_{\mu}$ for
$\mu\in I$ such that $L_{\mu}\neq L$. 

We need the following result of Kostant \cite{Kostant_Principal3D}.
\begin{prop}
(Kostant) Let $G$ be a connected split reductive group over $k$.
The list of numbers, with multiplicities, in 
\[
\left\{ \langle\rho_{B},\alpha^{\vee}\rangle+1:\alpha\in\Phi_{B}\right\} 
\]
is the same as the list of numbers, with multiplicities, in 
\[
\left\{ \langle\rho_{B},\alpha^{\vee}\rangle:\alpha\in\Phi_{B}-\Delta_{B}\right\} 
\]
together with positive integers $a_{1},\dots,a_{n}\ge2$ where $n$
is the cardinality of $\Delta_{B}$. 
\end{prop}

\begin{rem*}
The numbers $a_{1},\dots,a_{n}$ can be explicitly determined in terms
of the Poincar\'e polynomial of the Weyl group of $G$ (see Humphreys
\cite{Humphreys_CoxeterGroups}, chapters 1 and 3). However, we do
not need this. 
\end{rem*}
For the remainder of this section, we assume that $G$ is semi-simple
and we fix a \emph{maximal }parabolic subgroup $P\supset B$. Let
$\varpi\in\mathfrak{a}_{P}^{*}$ be the fundamental weight corresponding
to $P$. The standard Levi decomposition $P=N\rtimes M$ gives 
\[
\Phi_{B}=\Phi_{M}^{+}\sqcup\Phi_{N}
\]
where $\Phi_{M}^{+}$ and $\Phi_{N}$ are the roots with root groups
in $M\cap B$ and $N$ respectively. In particular, $\Delta_{B}^{P}\subset\Phi_{M}^{+}$.
Let $\Lambda=\Lambda_{P}+\Lambda_{B}^{P}\in\mathfrak{a}_{B}^{*}\otimes\mathbb{C}$
as in subsection \ref{subsec:DecomLambda}. 
\begin{prop}
\label{prop:franke}Let $\Lambda_{P}=s\varpi$ and $\rho_{B}^{P}+s\varpi$
for $s\in\mathbb{C}$ be a parametrization of $S_{P}=\rho_{B}^{P}+\mathfrak{a}_{P}^{*}\otimes\mathbb{C}$.
The poles of the Eisenstein series $E_{\varpi s}^{P}(g)$ in the region
$\mathsf{Re}(s)\ge0$ are real and contained in the set 
\[
\left\{ S_{P}\cap S_{\alpha}:\alpha\in\Phi_{N}\right\} 
\]
\end{prop}

\begin{proof}
We prove the result in several steps. We begin with some preliminary
remarks. 

By the general theory of Eisenstein series, it is enough to prove
this for the constant term $c_{B}E_{\Lambda_{P}}^{P}$, or the equivalently
the corresponding statement for
\[
\mathsf{Res}_{S_{P}}c_{B}E_{\bullet}^{B}(g)=\mathsf{Res}_{S_{P}}\left(\sum_{w\in W}c_{w,\Lambda}a(g)^{\rho_{B}+w\cdot\Lambda}\right)
\]
The term $\mathsf{Res}_{S_{P}}c_{w,\Lambda}\neq0$ only if $w\cdot\Delta_{B}^{P}\subset-\Phi_{B}$.
Fix a $w\in W$ satisfying this property. 

The product formula for $c_{w,\Lambda}$ contains terms $\xi(1+\langle\Lambda,\gamma^{\vee}\rangle)$
that can contribute poles from the critical zeros only if $\text{Re}\langle\Lambda,\gamma^{\vee}\rangle\in(-1,0)$.
This happens at $\Lambda=\rho_{B}^{P}+s\varpi\in S_{P}$ if 
\[
-1<\langle\rho_{B}^{P},\gamma^{\vee}\rangle+\langle\varpi,\gamma^{\vee}\rangle\text{Re}(s)<0
\]
Note that $\langle\varpi,\gamma^{\vee}\rangle\ge0$ for all positive
roots $\gamma$. In the region $\text{Re}(s)\ge0$, the term $\xi(1+\langle\rho_{B}^{P},\gamma^{\vee}\rangle+s\langle\varpi,\gamma^{\vee}\rangle)$
can contribute poles from the critical zeros only if $\langle\rho_{B}^{P},\gamma^{\vee}\rangle<0$. 

We prove the theorem by showing that if $w\cdot\gamma<0$ so that
it occurs in the product formula of $c_{w,\Lambda}$ and $\langle\rho_{B}^{P},\gamma^{\vee}\rangle<0$,
it is cancelled as in example \ref{subsec:SL3} above for $SL_{3}$. 

\textbf{Step 1:} We first deal with $\gamma^{\vee}$ for $\gamma\in\Phi_{M}^{+}$.
From $w\cdot\Delta_{B}^{P}\subset-\Phi_{B}$, we know that $w$ maps
all the roots in $\Phi_{M}^{+}$ to negative roots. For $\gamma\in\Phi_{M}^{+}$,
\[
\langle\Lambda_{P},\gamma^{\vee}\rangle=0,\quad\text{ for }\Lambda_{P}\in\mathfrak{a}_{P}^{*}\otimes\mathbb{C}
\]
On $S_{P}=\rho_{B}^{P}+\mathfrak{a}_{P}^{*}\otimes\mathbb{C}$, we
have
\[
\langle\rho_{B}^{P}+\Lambda_{P},\gamma^{\vee}\rangle=\langle\rho_{B}^{P},\gamma^{\vee}\rangle,\quad\text{ for }\gamma\in\Phi_{M}^{+},\Lambda_{P}\in S_{P}
\]
By applying the above result of Kostant to $M$, we get
\[
\frac{\prod_{\alpha\in\Phi_{M}^{+}-\Delta_{B}^{P}}\xi(\langle\rho_{B}^{P},\alpha^{\vee}\rangle)}{\prod_{\alpha\in\Phi_{M}^{+}}\xi(1+\langle\rho_{B}^{P},\alpha^{\vee}\rangle)}\bigg|_{S_{P}}=\frac{1}{\prod_{i=1}^{p}\xi(a_{i})}\neq0
\]
for some integers $a_{1},\dots,a_{p}\ge2$. 

Thus, in the product for $\mathsf{Res}_{S_{P}}c_{w,\Lambda}\neq0$,
we need only concern with poles from the critical zeros of terms $\xi(1+\langle\Lambda,\gamma^{\vee}\rangle)$
for $\gamma\in\Phi_{N}$. We do this in the next few steps. 

\textbf{Step 2: }We now characterize the roots that may contribute
poles from the critical zeros of $\xi$ in terms of the structure
of $P$. Let $H,X,Y$ be the standard notation for the principal $\mathfrak{sl}_{2}$
triple in $^{L}\mathfrak{m}$. For $\gamma\in\Phi_{N}$, the one-dimensional
space $^{L}\mathfrak{g}_{\gamma^{\vee}}$ is a $H$-eigenspace with
eigenvalue $\langle2\rho_{B}^{P},\gamma^{\vee}\rangle$ under the
adjoint action. We need to prove a cancellation for $\gamma^{\vee}$
with negative $H$-eigenvalue. 

\textbf{Step 3: }Now we obtain the coroots that give cancellation
to the troublesome roots of the previous step. To do this, let 
\[
^{L}\mathfrak{n}=r_{1}\oplus r_{2}\oplus\cdots\oplus r_{m}
\]
where $r_{1},\dots,r_{m}$ are the irreducible constituents of the
adjoint representation of $^{L}M$ on $^{L}\mathfrak{n}$ as described
in \ref{subsec:LanglandsComp}. We have the $H$-eigenspace decomposition
\[
r_{j}=\bigoplus_{\ell\in\mathbb{Z}}W_{\ell}(j),\quad W_{\ell}(j):=\left\{ v\in r_{j}:H\cdot v=\ell v\right\} 
\]
Fix $j\in\{1,\dots,m\}$ and $k<0$. Let 
\[
\Gamma_{k}(j):=\left\{ \gamma\in\Phi_{B}:{}^{L}\mathfrak{g}_{\gamma^{\vee}}\subset W_{k}(j)\text{ and }w\cdot\gamma<0\right\} 
\]
For any $\delta^{\vee}=\gamma^{\vee}+\theta^{\vee}$ for some $\gamma\in\Gamma_{k}(j)$
and $\theta\in\Delta_{B}^{P}$, we have $w\cdot\delta<0$ since 
\[
\delta=\frac{|\gamma^{\vee}|^{2}}{|\delta^{\vee}|^{2}}\gamma+\frac{|\theta^{\vee}|^{2}}{|\delta^{\vee}|^{2}}\theta\quad\text{and}\quad w\cdot\gamma,\ w\cdot\theta<0.
\]
By lemma \ref{lem:X-action},
\[
X\cdot\left(\bigoplus_{\gamma\in\Gamma_{k}(j)}{}^{L}\mathfrak{g}_{\gamma^{\vee}}\right)\subset\bigoplus_{\delta\in\Gamma_{k+2}(j)}{}^{L}\mathfrak{g}_{\delta^{\vee}}
\]
Since $W_{k}(j)$ is an $H$-eigenspace of \emph{negative} eigenvalue,
the action of $X$ is injective and the cardinalities satisy
\[
\#\Gamma_{k}(j)\le\#\Gamma_{k+2}(j)
\]

\textbf{Step 4: }We now prove the cancellation for $\gamma\in\Gamma_{k}(j)$
when $k<0$. Let $\delta^{\vee}=\gamma^{\vee}+\theta^{\vee}$ for
some $\gamma\in\Gamma_{k}(j)$ and $\theta\in\Delta_{B}^{P}$. We
have
\[
\langle\rho_{B}^{P}+\varpi s,\delta^{\vee}\rangle=\langle\rho_{B}^{P},\delta^{\vee}\rangle+js
\]
\[
=\langle\rho_{B}^{P},\gamma^{\vee}\rangle+\langle\rho_{B}^{P},\theta^{\vee}\rangle+js
\]
\[
=\frac{k}{2}+1+js
\]
Thus,
\[
\frac{\xi\left(\langle\rho_{B}^{P}+s\varpi,\delta^{\vee}\rangle\right)}{\xi\left(\langle\rho_{B}^{P}+s\varpi,\gamma^{\vee}\rangle+1\right)}=\frac{\xi(\frac{k}{2}+1+js)}{\xi(\frac{k}{2}+1+js)}=1
\]
By $\#\Gamma_{k}(j)\le\#\Gamma_{k+2}(j)$ (for $k<0$), we have the
required cancellation. We have shown that for $\text{Re}(s)\ge0$,
the critical zeros of $\xi$ do not yield poles of $E_{s\varpi}^{P}$. 

The remaining poles are given by 
\[
S_{P}\cap S_{\gamma}\quad(\gamma\in\Phi_{N})
\]
by the remark about residues at the beginning of this subsection. 
\end{proof}
\begin{rem*}
Similar arguments appear in justifying the contour deformation of
Langlands (see \cite{Opdam_and_Others}\cite{Moeglin_Compositio,Moeglin-Waldspurger-89})
and is, in principle, known in great generality \cite{FrankeResidual}.
However, this cancellation is \emph{not }sufficient to determine the
poles of degenerate Eisenstein series, even for $SL_{3}$ as shown
in \ref{subsec:SL3}. 
\end{rem*}

\subsection{A simple function determining the poles of $E_{\bullet}^{P}(g)$}

Next we show that the poles of $E_{s\varpi}^{P}$ in the region $\text{Re}(s)\ge0$
are determined by a simple fraction with the numerator and the denominators
from the zeros and the poles of $E_{\Lambda}^{B}$. 
\[
E_{\Lambda}^{B}(g)=\prod_{\alpha\in\Phi_{B}}\frac{\langle\Lambda,\alpha^{\vee}\rangle}{\langle\Lambda,\alpha^{\vee}\rangle-1}\cdot E_{\Lambda}^{*}(g)
\]

\begin{prop}
\label{prop:reduction}Let 
\[
F(s)=\frac{\prod_{\alpha\in\Phi_{B}}\langle\rho_{B}^{P}+s\varpi,\alpha^{\vee}\rangle}{\prod_{\alpha\in\Phi_{B}-\Delta_{B}^{P}}\left(\langle\rho_{B}^{P}+s\varpi,\alpha^{\vee}\rangle-1\right)},
\]
For $\text{Re}(s)\ge0$, the map $s\mapsto F(s)^{-1}\cdot E_{s\varpi}^{P}(g)$
is a holomorphic function taking values in the space of smooth functions
of uniform moderate growth on $G(F)\backslash G(\mathbb{A})$. 
\end{prop}

\begin{proof}
By the general theory of Eisenstein series, it is enough to show this
for $c_{B}\left(F(s)^{-1}E_{s\varpi}^{P}\right)$. From proposition
\ref{prop:franke}, we know that the poles can occur only along the
intersections $S_{\alpha}\cap S_{P}$ ($\alpha\in\Phi_{N}$). 

The function $c_{B}E_{\Lambda}^{B}(g)$ has a simple zero along the
hyperplanes 
\[
H_{\alpha}=\left\{ \Lambda\in\mathfrak{a}_{B}^{*}\otimes\mathbb{C}:\langle\Lambda,\alpha^{\vee}\rangle=0\right\} ,\quad\text{(}\alpha\in\Phi_{B}\text{)},
\]
and has simple poles along $S_{\alpha}$ for $\alpha\in\Phi_{B}$.
We have
\[
c_{B}E_{\Lambda}^{B}(g)=\prod_{\alpha\in\Phi_{B}}\frac{\langle\Lambda,\alpha^{\vee}\rangle}{\langle\Lambda,\alpha^{\vee}\rangle-1}\cdot c_{B}E_{\Lambda}^{*}(g)
\]
Then $c_{B}E_{\Lambda}^{*}(g)$ extends to a meromorphic function
of $S_{P}$ and is \emph{holomorphic }at\emph{ }$\rho_{B}^{P}+s\varpi$
for $\text{Re}(s)\ge0$ by proposition \ref{prop:franke}. We have
(the constant $c\neq0$ below is from Langlands' residue formula \ref{prop: residue}),
\[
c\cdot c_{B}E_{s\varpi}^{P}(g)=\mathsf{Res}_{S_{P}}c_{B}E_{\Lambda}^{B}(g)
\]
\[
=\mathsf{Res}_{S_{P}}\left(\prod_{\alpha\in\Phi_{B}}\frac{\langle\Lambda,\alpha^{\vee}\rangle}{\langle\Lambda,\alpha^{\vee}\rangle-1}\cdot c_{B}E_{\Lambda}^{*}(g)\right)
\]
\[
=\mathsf{Res}_{S_{P}}\left(\prod_{\alpha\in\Phi_{B}}\frac{\langle\Lambda,\alpha^{\vee}\rangle}{\langle\Lambda,\alpha^{\vee}\rangle-1}\right)\cdot c_{B}E_{\Lambda}^{*}(g)\bigg|_{S_{P}}
\]
Note that
\[
\mathsf{Res}_{S_{P}}\prod_{\alpha\in\Phi_{B}}\frac{\langle\Lambda,\alpha^{\vee}\rangle}{\langle\Lambda,\alpha^{\vee}\rangle-1}=\frac{\prod_{\alpha\in\Phi_{B}}\langle\Lambda,\alpha^{\vee}\rangle}{\prod_{\alpha\in\Phi_{B}-\Delta_{B}^{P}}\langle\Lambda,\alpha^{\vee}\rangle-1}\bigg|_{S_{P}}
\]
\[
=\frac{\prod_{\alpha\in\Phi_{B}}\langle\rho_{B}^{P}+s\Lambda_{P},\alpha^{\vee}\rangle}{\prod_{\alpha\in\Phi_{B}-\Delta_{B}^{P}}\langle\rho_{B}^{P}+\Lambda_{P},\alpha^{\vee}\rangle-1}=\frac{\prod_{\alpha\in\Phi_{B}}\langle\rho_{B}^{P}+s\varpi,\alpha^{\vee}\rangle}{\prod_{\alpha\in\Phi_{B}-\Delta_{B}^{P}}\langle\rho_{B}^{P}+\varpi,\alpha^{\vee}\rangle-1}=F(s)
\]
Thus, 
\[
c\cdot F(s)^{-1}\cdot E_{s\varpi}^{P}=E_{\rho_{B}^{P}+s\varpi}^{*}
\]
The function $E_{\rho_{B}^{P}+s\varpi}^{*}$ is holomorphic for $\text{Re}(s)\ge0$. 
\end{proof}
\begin{rem}
\label{rem:CriticalZeros}Already the $\text{SL}_{2}$ Eisenstein
series $E_{s}(z)$ in the introduction has poles from the critical
zeros of $\xi(s)$ when $\text{Re}(s)<0$, since poles of $c(s)=\xi(s)/\xi(1+s)$
are poles of $E_{s}(z)$. However, these poles do not play a role
in contour deformation of Langlands (see Godement \cite{Godement_Boulder}
and Moeglin \cite{Moeglin_ICM}). 
\end{rem}

\subsection{\label{subsec:cancellationExplicated}Main theorem}

Now we prove the main theorem of this paper:
\begin{thm}
Let $G$ be a split semi-simple linear algebraic group over a number
field and $P=N\rtimes M$ be the standard Levi decomposition of a
standard maximal parabolic subgroup $P$. Let 
\[
^{L}\mathfrak{n}=r_{1}\oplus r_{2}\oplus\cdots\oplus r_{m}
\]
where $r_{1},\dots,r_{m}$ are the irreducible constituents of the
adjoint representation of $^{L}M$ on $^{L}\mathfrak{n}$ as described
in \ref{subsec:LanglandsComp}. Let 
\[
r_{j}\simeq\bigoplus_{\ell\ge0}V_{\ell}^{m_{\ell}(j)},\quad V_{k}=\text{sym}^{k}(\text{std})
\]
be the decomposition into irreducible constituents of $r_{j}$ under
the action of the principal $\mathfrak{sl}_{2}\mathbb{C}\subset{}^{L}\mathfrak{m}$.
Let 
\[
p(s)=\prod_{j=1}^{m}\prod_{\ell\ge0}(js-1-\ell/2)^{m_{\ell}(j)}\ \in\mathbb{C}[s]
\]

In the region $\mathsf{Re}(s)\ge0$, $p(s)\cdot E_{s\varpi}^{P}(g)$
is holomorphic. Is not identically zero as a function on $G(\mathbb{A})$
when $\text{Re}(s)>0$. 
\end{thm}

\begin{proof}
We use proposition \ref{prop:reduction}. Let 
\[
F(s)=\frac{\prod_{\alpha\in\Phi_{B}}\langle\rho_{B}^{P}+s\varpi,\alpha^{\vee}\rangle}{\prod_{\alpha\in\Phi_{B}-\Delta_{B}^{P}}\langle\rho_{B}^{P}+\varpi,\alpha^{\vee}\rangle-1}
\]
We show that there are lots of cancellations as in the $\text{SL}_{3}$
case \ref{subsec:SL3}. The decomposition $P=N\rtimes M$ gives 
\[
\Phi_{B}=\Phi_{M}^{+}\sqcup\Phi_{N}
\]
where $\Phi_{M}^{+}$ and $\Phi_{N}$ are the roots with root groups
in $M\cap B$ and $N$ respectively. 

\textbf{Step 1: }Terms from $\Phi_{M}^{+}$ do not contribute. For
any $\eta\in\Phi_{M}^{+}$, the term $\langle\rho_{B}^{P}+\Lambda_{P},\eta^{\vee}\rangle=\langle\rho_{B}^{P},\eta^{\vee}\rangle$
is a non-zero constant. If $\eta\notin\Delta_{B}^{P}$, then $\langle\rho_{B}^{P},\eta^{\vee}\rangle>1$
and 
\[
\frac{\prod_{\eta\in\Phi_{M}^{+}}\langle\rho_{B}^{P}+\Lambda_{P},\eta^{\vee}\rangle}{\prod_{\eta\in\Phi_{M}^{+}-\Delta_{B}^{P}}\langle\rho_{B}^{P}+\Lambda_{P},\eta^{\vee}\rangle-1}=\text{constant}\ c>0
\]
Thus, only the terms from $\Phi_{N}$ need to be considered. 

\textbf{Step 2: }Grouping the terms in $\Phi_{N}$ for cancellation
in the next step. Let

\[
^{L}\mathfrak{n}=r_{1}\oplus r_{2}\oplus\cdots\oplus r_{m}
\]
and
\[
\Gamma_{j}:=\left\{ \gamma^{\vee}\in\Phi_{B}^{\vee}:{}^{L}\mathfrak{g}_{\gamma^{\vee}}\subset r_{j}\right\} 
\]
That is, $\langle\varpi,\gamma^{\vee}\rangle=j$ for all $\gamma^{\vee}\in\Gamma_{j}$.
Let $\{H,X,Y\}$ be the principal $\mathfrak{sl}_{2}\mathbb{C}$ triple
in $^{L}\mathfrak{m}$. Note that 
\[
r_{j}=\bigoplus_{\gamma\in\Gamma_{j}}{}^{L}\mathfrak{g}_{\gamma^{\vee}}
\]
is a decomposition of $r_{j}$ into $H$-eigenspaces (although not
under the full $\mathfrak{sl}_{2}\mathbb{C}$, see example \ref{exa:GL4}). 

\textbf{Step 3: }Cancellations. For $\gamma\in\Gamma_{j}$ and $\Lambda=\rho_{B}^{P}+s\varpi\in S_{P}$,
we have
\[
\langle\Lambda,\gamma^{\vee}\rangle=\langle\rho_{B}^{P},\gamma^{\vee}\rangle+\langle s\varpi,\gamma^{\vee}\rangle=\langle\rho_{B}^{P},\gamma^{\vee}\rangle+js
\]
Thus, only $H$ eigenvalues $\langle\rho_{B}^{P},\bullet\rangle$
play a role in the cancellation. For roots $\gamma_{0},\gamma_{1},\dots,\gamma_{k}\in\Gamma$
with 
\[
\left\{ \langle2\rho_{B}^{P},\gamma_{i}^{\vee}\rangle:i=0,\dots,k\right\} =\{-k,-k+2,\dots,k\},
\]
we have
\[
\prod_{i=0}^{k}\frac{\langle\rho_{B}^{P}+s\varpi,\gamma_{i}^{\vee}\rangle}{\langle\rho_{B}^{P}+s\varpi,\gamma_{i}^{\vee}\rangle-1}=\frac{js+k/2}{js-k/2-1}
\]

\textbf{Step 4: }Conclusion. We have 
\[
F(s)=\frac{\prod_{j=1}^{m}\prod_{\ell\ge0}(js+\ell/2)^{m_{\ell}(j)}}{\prod_{j=1}^{m}\prod_{\ell\ge0}(js-1-\ell/2)^{m_{\ell}(j)}}\ \in\mathbb{C}(s)
\]
Since the terms in the numerator vanish only for $s\le0$ and the
terms in the denominator vanish only for $s>0$, there can be no further
cancellation. If $p(s)$ is the denominator in the above expression
of $F(s)$, then by proposition \ref{prop:reduction}, we have $p(s)\cdot E_{s\varpi}^{P}(g)$
is holomorphic for $\text{Re}(s)\ge0$. Since the numerator of $F(s)\neq0$
when $\text{Re}(s)>0$, the non-vanishing assertion follows from the
non-vanishing of $E_{\rho_{B}^{P}+s\varpi}^{*}$ which itself follows
from the simplicity of the zeros of $E_{\Lambda}^{B}$ . 
\end{proof}

\subsection{\label{subsec:Computational-procedure}Explicit computations illustrated
for $A_{n}$ and $G_{2}$}

Let $P$ be a maximal parabolic subgroup of $G$ with $\Theta:=\Delta_{B}^{P}=\Delta_{B}-\{\beta\}$
for some $\beta\in\Delta_{B}$. To compute the poles of $E_{\Lambda}^{P}(g)$
in the positive tube, we need to decompose $^{L}\mathfrak{n}$ under
the action of the principal $\mathfrak{sl}_{2}\mathbb{C}$ in $^{L}\mathfrak{m}$. 

Let $H,X,Y$ be the standard $\mathfrak{sl}_{2}$ triple for the principal
$\mathfrak{sl}_{2}\mathbb{C}$ in $^{L}\mathfrak{m}$. Let $\gamma^{\vee}$
be such that 
\[
\gamma^{\vee}=\cdots+j\beta^{\vee}+\cdots
\]
Then $^{L}\mathfrak{g}_{\gamma^{\vee}}\subset r_{j}$. Since $\langle2\rho_{B}^{P},\theta^{\vee}\rangle=2$
for $\theta\in\Theta$, if $\gamma^{\vee}+\theta^{\vee}$ is a coroot,
then the $H$-eigenvalue corresponding to $\gamma^{\vee}+\theta^{\vee}$
is $2+\langle2\rho_{B}^{P},\gamma^{\vee}\rangle$. Thus, one can decompose
$^{L}\mathfrak{n}$ by counting the number of roots $\gamma^{\vee}$
with $^{L}\mathfrak{g}_{\gamma^{\vee}}\subset r_{j}$ of a given height.
This gives us a list of $H$-eigenvalues on $r_{j}$, which is enough
to decompose $r_{j}$ abstractly in terms of symmetric powers under
the principal $\mathfrak{sl}_{2}\mathbb{C}$ action. 
\begin{rem}
Note that ordering of roots by height and the $j$s occurring in the
decomposition of $^{L}\mathfrak{n}$ does not depend on the lattice
of characters. Thus, the result only depends on the \emph{root system
}of the dual group. That is, the results are identical for different
isogeny classes of $G$. 
\end{rem}

\begin{example}
We illustrate the procedure for $A_{n}$ with the Dynkin diagram labelling:\begin{center} 
\texttt{
$A_n$: \dynkin[labels = {1,2,,n}, edge length=1cm]A{} \newline 
}
\end {center}Let $\alpha_{1},\dots,\alpha_{n}$ be the corresponding simple roots.
The positive roots are of the form 
\[
r_{ij}:=\alpha_{i}+\dots+\alpha_{j},\quad1\le i\le j\le n
\]
Height of $r_{ij}=j-i+1$. There are $p_{i}:=n-i+1$ roots of height
$i$. 

Let $P=N\rtimes M$ be the standard Levi decomposition of a standard
\emph{maximal }parabolic subgroup $P$. Thus $\Delta_{B}-\Delta_{B}^{P}$
has cardinality one. Let $^{L}M$ be of type $A_{a}\times A_{b}$
where $a+b=n-1$. The number $m_{i}$ of positive roots of height
$i$ with root spaces in $^{L}\mathfrak{m}$ is 
\[
m_{i}=\begin{cases}
(a-i+1)+(b-i+1)=n-2i+1 & i\le\min\{a,b\}\\
\max\{a,b\}-i+1 & \min\{a,b\}<i\le\max\{a,b\}\\
0 & i>\max\{a,b\}
\end{cases}
\]
The number $n_{i}$ of roots of height $i$ with root spaces contained
in $^{L}\mathfrak{n}$ is 
\[
n_{i}=p_{i}-m_{i}=\begin{cases}
i & i\le\min\{a,b\}\\
n-\max\{a,b\}=\min\{a,b\}+1 & \min\{a,b\}<i\le\max\{a,b\}\\
n-i+1 & i>\max\{a,b\}
\end{cases}
\]
Thus, 
\[
^{L}\mathfrak{n}\simeq\bigoplus_{k=n-1-2\min\{a,b\}}^{(n-1)}V_{k},\quad k\text{ increments of }2
\]
When $\text{Re}(s)\ge0$, the Eisenstein series $E_{s\varpi}^{P}(g)$
has simple poles for
\[
s\in\left\{ 1+\frac{n-1}{2},\cdots,1+\frac{n-1-2\min\{a,b\}}{2}\right\} 
\]
or 
\[
s=\left\{ \frac{n+1}{2},\frac{n-1}{2},\cdots,\frac{n+1}{2}-\min\{a,b\}\right\} 
\]
\end{example}

\begin{example}
This example illustrates some features not present in type $A_{n}$.
Let $G=G_{2}$, the exceptional $k$-group split over $k$. Let $\alpha$
be the short simple root and $\beta$ be the long simple root of $G_{2}$.
The positive roots are 
\[
\Phi_{B}=\left\{ \alpha,\ \beta,\ \alpha+\beta,\ 2\alpha+\beta,\ 3\alpha+\beta,\ 3\alpha+2\beta\right\} 
\]
There are two maximal parabolic subgroups $P=N_{P}\rtimes M_{P}$
and $Q=N_{Q}\rtimes M_{Q}$. Let $\Delta_{B}^{P}=\{\alpha\}$ and
$\Delta_{B}^{Q}=\{\beta\}$.

Note that in $^{L}G\simeq G_{2}$, the root $\alpha^{\vee}$ is the
long simple root and the root $\beta^{\vee}$ is the short simple
root. 

For $P$: We have 
\[
^{L}\mathfrak{n}_{P}=r_{1}\oplus r_{2}\oplus r_{3},\quad r_{1},r_{3}\simeq V_{1}\text{ and }r_{2}\simeq V_{0}
\]
and
\[
s=\frac{3}{2},\frac{1}{2}\left(1+\frac{0}{2}\right),\frac{1}{3}\left(1+\frac{1}{2}\right)
\]
That is, $E_{s}^{P}(g)$ has a simple pole at $s=\frac{3}{2}$ and
a double pole at $s=\frac{1}{2}$. 

For $Q$: We have 
\[
^{L}\mathfrak{n}_{Q}=r_{1}\oplus r_{2},\quad r_{1}\simeq V_{3},\ r_{2}\simeq V_{0}
\]
and
\[
s=1+\frac{3}{2},\frac{1}{2}\left(1+\frac{0}{2}\right)
\]
That is, $E_{s}^{Q}(g)$ has a simple poles at $s=\frac{5}{2},\frac{1}{2}$.
\end{example}

\begin{rem*}
These results have been known at least since 1976 (see appendix III
of Langlands \cite{Langlands-Spectral}). 
\end{rem*}

\section{\label{sec:Classical}Computations for classical groups}

Throughout this section, we denote by $V_{k}$ the $(k+1)$-dimensional
irreducible representation of $\mathfrak{sl}_{2}\mathbb{C}$. We begin
by recalling the following fact.
\begin{fact}
\label{fact:alpha(H)=00003D2}For a semi-simple Lie algebra $\mathfrak{g}$
with a chosen basis of simple roots $\Delta$, the principal $\mathfrak{sl}_{2}=\text{span}\{H,X,Y\}$
satisfies 
\[
\alpha(H)=2\quad\text{for all }\alpha\in\Delta.
\]
\end{fact}

This characterization allows us to explicitly write $H$ in all our
computations using the data given in Bourbaki \cite{Bourbaki_LieTheory}
as appendices (called ``plates''). 

Throughout this section, let $P=N\rtimes M$ be the standard Levi
decomposition of a standard \emph{maximal} parabolic subgroup $P$.
To compute the poles of $E_{\Lambda}^{P}(g)$ in the right-half space,
we need to compute the decomposition of $^{L}\mathfrak{n}$ under
the action of the principal $\mathfrak{sl}_{2}$ with standard triples
$\{H,X,Y\}$ in $^{L}\mathfrak{m}$. Let $\Delta=\{\alpha_{1},\dots,\alpha_{n}\}$. 

We write $(\mu;m)$ to indicate that the Eisenstein series has a pole
of order $m$ at $s=\mu$. If the pole is simple, we only write $\mu$. 

\subsection{Description of the groups}

Let $k$ be a number field. We view $V=k^{n}$ as column vectors.
For an $n\times n$ matrix $\Omega$, let 
\[
G_{\Omega}(V)=\left\{ g\in GL(V):g^{T}\Omega g=\Omega\right\} 
\]
Let
\[
\text{SL}_{n}(V)=\left\{ g\in GL(V):\det g=1\right\} 
\]
Now we define the isometry groups. Let 
\[
\omega_{n}=\begin{bmatrix}0 & 0 & \cdots & 0 & 1\\
0 &  &  & 1 & 0\\
\vdots &  & \iddots &  & \vdots\\
0 & 1 &  &  & 0\\
1 & 0 & \cdots & 0 & 0
\end{bmatrix}\quad n\times n\text{ matrix}
\]
Let
\[
\text{Sp}_{2n}(V)=G_{\Omega}(V),\quad V=k^{2n},\ \Omega=\left(\begin{matrix}0 & \omega_{n}\\
-\omega_{n} & 0
\end{matrix}\right)
\]
and (for $n\ge2$)
\[
\text{SO}_{n}(V)=G_{\Omega}(V)\bigcap\text{SL}(V),\quad V=k^{n},\ \Omega=\omega_{n}
\]
We take the standard choice of maximal isotropic flags to define our
Borel subgroup as upper-triangular matrices in these groups. 

\subsection{$B_{n}$: odd special orthogonal groups}

The Dynkin diagram is 

\begin{center} 
\texttt{
$B_n$: \dynkin[labels = {1,2,,n-1,n}, edge length=1cm]B{} \newline 
}
\end {center}The dual group is $\text{Sp}_{2n}(\mathbb{C})$. 

\subsubsection{Siegel parabolic case}

For $\Delta-\Delta_{B}^{P}=\{\alpha_{n}\}$, the derived group of
$^{L}M$ is of type $A_{n-1}$. The corresponding $H$ is 

\[
H=\text{diag}(n-1,n-3,\dots,1-n,n-1,n-3,\dots,1-n)
\]
Note that $\alpha(H)=2$ for all $\alpha\in\Delta_{B}^{P}$. We identify
$^{L}\mathfrak{n}$ with $n\times n$ matrices symmetric about the
non-principal diagonal. Then 
\[
^{L}\mathfrak{n}=r_{1}\simeq V_{2(n-1)}\oplus V_{2(n-3)}\oplus\cdots\oplus\begin{cases}
V_{0} & n\text{ odd}\\
V_{2} & n\text{ even}
\end{cases}
\]
Poles at 
\[
s=n,n-2,n-4,\dots,\begin{cases}
1 & n\text{ odd}\\
2 & n\text{ even}
\end{cases}.
\]

\begin{rem}
\label{rem:B_n-case}We note that if $m\in M(\mathbb{A})\simeq GL_{n}(\mathbb{A})$,
then 
\[
e^{\langle\varpi,H_{P}(m)\rangle}=|\det m|^{s/2}
\]
IA common parametrization is $|\det m|^{s}$ in which case we have
simple poles at 
\[
s=\frac{n}{2},\frac{n}{2}-1,\frac{n}{2}-2,\dots,\begin{cases}
1/2 & n\text{ odd}\\
1 & n\text{ even}
\end{cases}
\]
\end{rem}

\subsubsection{Non-Siegel parabolic subgroups}

For $\Delta-\Delta_{B}^{P}=\{\alpha_{a}\}$ for $1\le a<n$, the derived
group of $^{L}M$ is of type $A_{a-1}\times C_{b}$ with $a+b=n$
and $b\ge1$. The corresponding $H$ is 

\[
H=(a-1,\dots,1-a,2b-1,\dots,3,1,-1,-3,\dots,1-2b,a-1\dots,1-a)
\]
using fact \ref{fact:alpha(H)=00003D2}. We have 
\[
^{L}\mathfrak{n}=r_{1}\oplus r_{2}
\]

We parametrize $r_{2}$ with $a\times a$ matrices symmetric about
the non-principal diagonal. We have 
\[
r_{2}\simeq V_{2(a-1)}\oplus V_{2(a-3)}\oplus\cdots\oplus\begin{cases}
V_{0} & a\text{ odd}\\
V_{2} & a\text{ even}
\end{cases}
\]
and this contributes at a simple pole at the points 
\[
s=\frac{a}{2},\frac{a}{2}-1,\frac{a}{2}-1\dots,\begin{cases}
1/2 & a\text{ odd}\\
1 & a\text{ even}
\end{cases}
\]

We parametrize $r_{1}$ with $a\times2b$ matrices. The corresponding
$H$-eigenvalues are given by 
\[
\begin{bmatrix}a-2b & (a-2b)+2 & \cdots & (a+2b)-2\\
(a-2b)+2 &  &  & (a+2b)-4\\
\vdots &  & \iddots & \vdots\\
4-(a+2b) & 6-(a+2b)\\
2-(a+2b) & 4-(a+2b) & \cdots & 2b-a
\end{bmatrix}
\]
The structure is similar to the $A_{n}$ computation and we have 
\[
r_{1}\simeq\bigoplus_{k=a+2b-2-2\left(\min\{a,2b\}-1\right)}^{k=a+2b-2}V_{k},\quad k\text{ increments of }2
\]
This contributes a simple pole at 
\[
s=\frac{a+2b}{2},\frac{a+2b}{2}-1,\dots,\frac{a+2b}{2}-\left(\min\{a,2b\}-1\right)
\]

$E_{s\varpi}^{P}$ has double poles when
\[
\frac{a}{2}\ge\frac{a+2b}{2}-\left(\min\{a,2b\}-1\right)\iff\min\{a,2b\}\ge b+1
\]
at 
\[
s=\frac{a}{2},\dots,\frac{a+2b+2}{2}-\min\{a,2b\}
\]

\subsection{$C_{n}$: symplectic groups}

The Dynkin diagram is 

\begin{center} 
\texttt{
$C_n$: \dynkin[labels = {1,2,,n-1,n}, edge length=1cm]C{} \newline 
}
\end {center}

\subsubsection{Siegel parabolic case}

For $\Delta-\Delta_{B}^{P}=\{\alpha_{n}\}$, the derived group of
$^{L}M$ is of type $A_{n-1}$. The corresponding $H$ is 

\[
H=\text{diag}(n-1,n-3,\dots,1-n,0,n-1,n-3,\dots,1-n)
\]
We have 
\[
^{L}\mathfrak{n}=r_{1}\oplus r_{2}
\]
with $r_{1}\simeq V_{n-1}$ which gives a simple pole at $s=\frac{n+1}{2}$.
We can parametrized $r_{2}$ by $n\times n$ matrices skew-symmetric
about the non-principal diagonal and 
\[
r_{2}\simeq V_{2(n-2)}\oplus V_{2(n-4)}\oplus\cdots\oplus\begin{cases}
V_{0} & n\text{ even}\\
V_{2} & n\text{ odd}
\end{cases}
\]
We have simple poles at
\[
s=\frac{n+1}{2},\frac{n-1}{2},\frac{n-3}{2},\dots,\begin{cases}
\frac{1}{2} & n\text{ even}\\
1 & n\text{ odd}
\end{cases}
\]

\begin{rem*}
This result was proved by Kudla-Rallis \cite{Kudla_Rallis}. 
\end{rem*}

\subsubsection{Non-Siegel parabolic subgroups}

For $\Delta-\Delta_{B}^{P}=\{\alpha_{a}\}$ for $1\le a<n$, the derived
group of $^{L}M$ is of type $A_{a-1}\times B_{b}$ with $a+b=n$
and $b\ge1$. The corresponding $H$ is 

\[
H=(a-1,\dots,1-a,2b,\dots,2,0,-2,\dots,-2b,a-1\dots,1-a)
\]
using fact \ref{fact:alpha(H)=00003D2}. We have $^{L}\mathfrak{n}=r_{1}\oplus r_{2}$
where
\[
r_{2}\simeq V_{2(a-2)}\oplus V_{2(a-4)}\oplus\cdots\oplus\begin{cases}
V_{0} & a\text{ even}\\
V_{2} & a\text{ odd}
\end{cases}
\]
This contributes simple poles at 
\[
s=\frac{a-1}{2},\frac{a-3}{2},\dots,\begin{cases}
\frac{1}{2} & a\text{ even}\\
1 & a\text{ odd}
\end{cases}
\]

We parametrize $r_{1}$ with $a\times(2b+1)$ matrices. The corresponding
$H$-eigenvalues are given by 
\[
\begin{bmatrix}(a-2b)-1 & (a-2b)+1 & \cdots & (a+2b)-1\\
(a-2b)-3 &  &  & (a+2b)-3\\
\vdots &  & \iddots & \vdots\\
3-(a+2b) & 5-(a+2b)\\
1-(a+2b) & 3-(a+2b) & \cdots & (2b-a)+1
\end{bmatrix}
\]
Thus, 
\[
r_{1}\simeq\bigoplus_{k=a+2b-1-2\left(\min\{a,2b+1\}-1\right)}^{k=a+2b-1}V_{k},\quad k\text{ increments of }2
\]
This contributes simple poles at 
\[
s=\frac{a+2b+1}{2},\frac{a+2b+1}{2}-1,\dots,\frac{a+2b+1}{2}-\left(\min\{a,2b+1\}-1\right)
\]

We have double poles when 
\[
\frac{a-1}{2}\ge\frac{a+2b+1}{2}-\left(\min\{a,2b+1\}-1\right)\iff\min\{a,2b+1\}\ge b+2
\]
at 
\[
\frac{a-1}{2},\frac{a-3}{2},\dots,\frac{a+2b+3}{2}-\min\{a,2b+1\}
\]

\begin{rem*}
See Hanzer \cite{Hanzer_nonSiegel} for a more extensive discussion
of this case. 
\end{rem*}

\subsection{$D_{n}$: even special orthogonal groups}

We take $n\ge3$. The Dynkin diagram is 

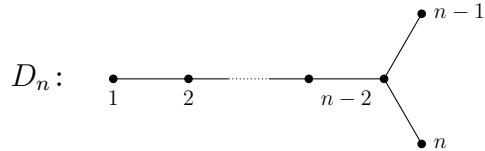
\begin{center} 
\texttt{
$D_n$: \dynkin[labels = {1,2,,n-2,n-1,n}, edge length=1cm]D{} \newline 
}
\end {center}The simple roots are denoted $\alpha_{1},\dots,\alpha_{n}$ with the
subscript corresponding to the labelled node in the Dynkin diagram.
The positive roots are $\alpha_{ij},\alpha'_{ij}$ and $\alpha'_{i}$,
where 
\[
\alpha_{ij}=\alpha_{i}+\cdots+\alpha_{j}\quad(1\le i\le j\le n-1)
\]
\[
\alpha'_{ij}=\alpha_{i}+\cdots+\alpha_{j-1}+2\alpha_{j}+\cdots+2\alpha_{n-2}+\alpha_{n-1}+\alpha_{n}\quad(1\le i<j\le n-1)
\]
\[
\alpha'_{i}=\alpha_{i}+\cdots+\alpha_{n-2}+\alpha_{n}\quad(1\le i\le n-1)
\]
The notation gives the convenience $\alpha'_{n-1}=\alpha_{n}$. 

The structure of parabolic subgroups of this case is a bit different.
There are two ``Siegel-type'' parabolic subgroups corresponding
to $\Delta-\Delta_{B}^{P}=\{\alpha_{n}\}\text{ or }\{\alpha_{n-1}\}$.
The derived group of the Levi in both cases in of type $A_{n-1}$.
We have $^{L}\text{SO}_{2n}=\text{SO}_{2n}(\mathbb{C})$.

\subsubsection{Siegel-type parabolic case}

For $\Delta-\Delta_{B}^{P}=\{\alpha_{n}\}$, the derived group of
$^{L}M$ is of type $A_{n-1}$. The corresponding $H$ is 

\[
H=\text{diag}(n-1,n-3,\dots,1-n,1-n,3-n,\dots,n-1)
\]
We identify $^{L}\mathfrak{n}$ with $n\times n$ matrices skew-symmetric
spanned by all $\{\alpha'_{ij}\}$ and $\{\alpha'_{i}\}_{i=1}^{n-1}$.
We have 
\[
^{L}\mathfrak{n}=r_{1}\simeq V_{2(n-2)}\oplus V_{2(n-4)}\oplus\cdots\oplus\begin{cases}
V_{0} & n\text{ even}\\
V_{2} & n\text{ odd}
\end{cases}
\]
We have simple poles at 
\[
s=n-1,n-3,n-5,\dots,\begin{cases}
1 & n\text{ even}\\
2 & n\text{ odd}
\end{cases}
\]

For $\Delta-\Delta_{B}^{P}=\{\alpha_{n}\}$, we have $^{L}\mathfrak{n}=r_{1}$
spanned by all $\{\alpha'_{ij}\}$ and $\{\alpha_{i,n-1}\}_{i=1}^{n-1}$.
Therefore the result is the same as the above case, namely, simple
poles at 
\[
s=n-1,n-3,n-5,\dots,\begin{cases}
1 & n\text{ even}\\
2 & n\text{ odd}
\end{cases}
\]
See remark \ref{rem:B_n-case}. 

\subsubsection{Non-Siegel parabolic subgroups}

For $\Delta-\Delta_{B}^{P}=\{\alpha_{a}\}$ for $1\le a\le n-2$,
the derived group of $^{L}M$ is of type $A_{a-1}\times D_{b}$ with
$a+b=n$ and $b\ge2$. The corresponding $H$ is 
\[
H=\left(a-1,\dots,1-a,2(b-1),\dots,2,0,0,-2,\dots,2(1-b),a-1,\dots,1-a\right)
\]
using fact \ref{fact:alpha(H)=00003D2}. We have $^{L}\mathfrak{n}=r_{1}\oplus r_{2}$
where
\[
r_{2}\simeq V_{2(a-2)}\oplus V_{2(a-4)}\oplus\cdots\oplus\begin{cases}
V_{0} & a\text{ even}\\
V_{2} & a\text{ odd}
\end{cases}
\]
This contributes simple poles at 
\[
s=\frac{a-1}{2},\frac{a-3}{2},\dots,\begin{cases}
\frac{1}{2} & a\text{ even}\\
1 & a\text{ odd}
\end{cases}
\]

We parametrize $r_{1}$ by two $a\times2b$ matrices with the corresponding
$H$-eigenvalues 
\[
\begin{bmatrix}a-2b+1 & \cdots & a-1 & a-1 & \cdots & a+2b-3\\
a-2b-1 &  & a-3 & a-3 &  & a+2b-5\\
\vdots &  & \vdots & \vdots &  & \vdots\\
3-a-2b & \cdots & 1-a & 1-a & \cdots & -a+2b-1
\end{bmatrix}
\]
The presence of two consecutive zeros in $H$ means we get an ``extra''
$V_{a-1}$. Writing it separately, 
\[
r_{1}\simeq V_{a-1}\oplus\bigoplus_{k=a+2b-3-2(\min\{a,2b-1\}-1)}^{a+2b-3}V_{k}\quad k\text{ increments of }2
\]
This contributes simple poles at 
\[
k=\frac{a+2b-1}{2},\dots,\frac{a+2b-1}{2}-\left(\min\{a,2b-1\}-1\right);\quad\frac{a+1}{2}
\]

We have poles at
\[
\frac{a+1}{2},\frac{a-1}{2},\dots,\frac{a+2b-1}{2}-\left(\min\{a,2b-1)-1\right)
\]
with double poles when 
\[
\frac{a+1}{2}\ge\frac{a+2b-1}{2}-\left(\min\{a,2b-1)-1\right)\iff\min\{a,2b-1\}\ge b
\]
at 
\[
\frac{a+1}{2},\dots,\frac{a+2b+1}{2}-\min\{a,2b-1\}.
\]

\section{\label{sec:Exceptional}Computations for exceptional Chevalley groups }

We use the standard numbering of roots for the exceptional groups
as in Bourbaki \cite{Bourbaki_LieTheory}. The ordering of roots by
height is available in Springer \cite{Springer_roots_by_height} (with
a different numbering). 

The Chevalley groups $E_{6},E_{7},$ and $E_{8}$ are self dual and
under $G\to{}^{L}G$ there is no relabelling of vertices in the corresponding
map between Dynkin diagrams. This is not true for $G_{2}$ and $F_{4}$. 

We write $(\mu;m)$ to indicate that the Eisenstein series has a pole
of order $m$ at $s=\mu$. If the pole is simple, we only write $\mu$.
The results of this section have been obtained using computers by
Segal and Halawi \cite{halawi2023poles}. 

\subsection{Type $E_{6}$}

The Dynkin diagram with standard numbering is 

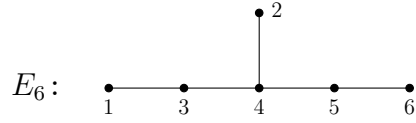
\begin{center} 
\texttt{
$E_6$: \centering \dynkin[label,edge length=1cm]E6 \newline 
}
\end {center}

\subsubsection{Poles for $P_{1},P_{6}$}

\[
r_{1}\simeq V_{4}\oplus V_{10}
\]
\[
s=3,6
\]

\subsubsection{Poles for $P_{3},P_{5}$}

\[
r_{1}\simeq V_{1}\oplus V_{3}\oplus V_{5}\oplus V_{7}
\]
\[
r_{2}\simeq V_{4}
\]
\[
s=\left(\frac{3}{2};2\right),\frac{5}{2},\frac{7}{2},\frac{9}{2}
\]

\subsubsection{Poles for $P_{4}$}

\[
r_{1}\simeq V_{1}^{\oplus2}\oplus V_{3}^{\oplus2}\oplus V_{5}
\]
\[
r_{2}\simeq V_{0}\oplus V_{2}\oplus V_{4}
\]
\[
r_{3}\simeq V_{1}
\]
\[
s=\left(\frac{1}{2};2\right),1,\left(\frac{3}{2};3\right),\left(\frac{5}{2};2\right),\frac{7}{2}
\]

\subsubsection{Poles for $P_{2}$}

\[
r_{1}\simeq V_{3}\oplus V_{5}\oplus V_{9}
\]
\[
r_{2}\simeq V_{0}
\]
\[
s=\frac{1}{2},\frac{5}{2},\frac{7}{2},\frac{11}{2}
\]

\subsection{Type $E_{7}$}

The Dynkin diagram with standard numbering is 

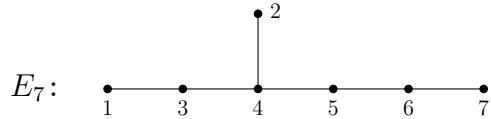
\begin{center} 
\texttt{
$E_7$: \centering \dynkin[label,edge length=1cm]E7 \newline 
}
\end {center}

\subsubsection{$P_{1}$}

\[
r_{1}\simeq V_{5}\oplus V_{9}\oplus V_{15}
\]
\[
r_{2}\simeq V_{0}
\]
\[
s=\frac{1}{2},\frac{7}{2},\frac{11}{2},\frac{17}{2}
\]

\subsubsection{$P_{2}$}

\[
r_{1}\simeq V_{0}\oplus V_{4}\oplus V_{6}\oplus V_{8}\oplus V_{12}
\]
\[
r_{2}\simeq V_{6}
\]
\[
s=1,2,3,4,5,7
\]

\subsubsection{$P_{3}$}

\[
r_{1}\simeq V_{1}\oplus V_{3}\oplus V_{5}\oplus V_{7}\oplus V_{9}
\]
\[
r_{2}\simeq V_{0}\oplus V_{4}\oplus V_{8}
\]
\[
r_{3}\simeq V_{1}
\]
\[
s=\left(\frac{1}{2};2\right),\left(\frac{3}{2};2\right),\left(\frac{5}{2};2\right),\frac{7}{2},\frac{9}{2},\frac{11}{2}
\]

\subsubsection{$P_{4}$}
\begin{lyxcode}
\[
r_{1}\simeq V_{0}\oplus V_{2}^{\oplus2}\oplus V_{4}^{\oplus2}\oplus V_{6}
\]
\[
r_{2}\simeq V_{2}^{\oplus2}\oplus V_{4}\oplus V_{6}
\]
\[
r_{3}\simeq V_{2}\oplus V_{4}
\]
\[
r_{4}\simeq V_{2}
\]
\[
s=\frac{1}{2},\frac{2}{3},\left(1;4\right),\frac{3}{2},(2;3),(3;2),4
\]
\end{lyxcode}

\subsubsection{$P_{5}$ }

\[
r_{1}\simeq V_{0}\oplus V_{2}\oplus V_{4}^{\oplus2}\oplus V_{6}\oplus V_{8}
\]
\[
r_{2}\simeq V_{2}\oplus V_{4}\oplus V_{6}
\]
\[
r_{3}\simeq V_{4}
\]
\[
s=\left(1;3\right),\frac{3}{2},\left(2;2\right),(3;2),4,5
\]

\subsubsection{$P_{6}$}

\[
r_{1}\simeq V_{3}\oplus V_{5}\oplus V_{9}\oplus V_{11}
\]
\[
r_{2}\simeq V_{0}\oplus V_{8}
\]
\[
s=\frac{1}{2},\left(\frac{5}{2};2\right),\frac{7}{2},\frac{11}{2},\frac{13}{2}
\]

\subsubsection{$P_{7}$}

\[
r_{1}\simeq V_{0}\oplus V_{8}\oplus V_{16}
\]
\[
s=\frac{1}{2},5,9
\]

\subsection{Type $E_{8}$}

The Dynkin diagram with standard numbering is 

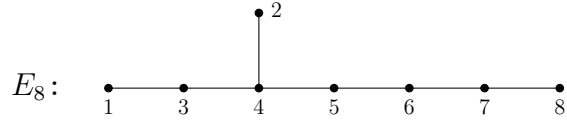
\begin{center} 
\texttt{
$E_8$: \centering \dynkin[label,edge length=1cm]E8 \newline 
}
\end {center}

\subsubsection{$P_{1}$}

\[
r_{1}\simeq V_{3}\oplus V_{9}\oplus V_{11}\oplus V_{15}\oplus V_{21}
\]
\[
r_{2}\simeq V_{0}\oplus V_{12}
\]
\[
s=\frac{1}{2},\frac{5}{2},\frac{7}{2},\frac{11}{2},\frac{13}{2},\frac{17}{2},\frac{23}{2}
\]

\subsubsection{$P_{2}$}

\[
r_{1}\simeq V_{3}\oplus V_{5}\oplus V_{7}\oplus V_{9}\oplus V_{11}\oplus V_{15}
\]
\[
r_{2}\simeq V_{0}\oplus V_{4}\oplus V_{8}\oplus V_{12}
\]
\[
r_{3}\simeq V_{7}
\]
\[
s=\frac{1}{2},\left(\frac{3}{2};2\right),\left(\frac{5}{2};2\right),\left(\frac{7}{2};2\right),\frac{9}{2},\frac{11}{2},\frac{13}{2},\frac{17}{2}
\]

\subsubsection{$P_{3}$}

\[
r_{1}\simeq V_{1}\oplus V_{3}\oplus V_{5}\oplus V_{7}\oplus V_{9}\oplus V_{11}
\]
\[
r_{2}\simeq V_{0}\oplus V_{4}\oplus V_{6}\oplus V_{8}\oplus V_{12}
\]
\[
r_{3}\simeq V_{5}\oplus V_{7}
\]
\[
r_{4}\simeq V_{6}
\]
\[
s=\frac{1}{2},1,\frac{7}{6},\left(\frac{3}{2};3\right),2,\left(\frac{5}{2};2\right),\left(\frac{7}{2};2\right),\frac{9}{2},\frac{11}{2},\frac{13}{2}
\]

\subsubsection{$P_{4}$}

\[
r_{1}\simeq V_{1}\oplus V_{3}^{\oplus2}\oplus V_{5}^{\oplus2}\oplus V_{7}
\]

\[
r_{2}\simeq V_{0}\oplus V_{2}\oplus V_{4}^{\oplus2}\oplus V_{6}\oplus V_{8}
\]

\[
r_{3}\simeq V_{1}\oplus V_{3}\oplus V_{5}\oplus V_{7}
\]

\[
r_{4}\simeq V_{2}\oplus V_{4}\oplus V_{6}
\]

\[
r_{5}\simeq V_{1}\oplus V_{3}
\]
\[
r_{6}\simeq V_{4}
\]
\[
s=\frac{3}{10},\left(\frac{1}{2};5\right),\frac{3}{4},\frac{5}{6},(1;2),\frac{7}{6},\left(\frac{3}{2};4\right),2,\left(\frac{5}{2};3\right),\left(\frac{7}{2};2\right),\frac{9}{2}
\]

\subsubsection{$P_{5}$}

\[
r_{1}\simeq V_{1}\oplus V_{3}^{2}\oplus V_{5}^{2}\oplus V_{7}\oplus V_{9}
\]
\[
r_{2}\simeq V_{0}\oplus V_{2}\oplus V_{4}^{2}\oplus V_{6}\oplus V_{8}
\]
\[
r_{3}\simeq V_{1}\oplus V_{3}\oplus V_{5}\oplus V_{7}
\]
\[
r_{4}\simeq V_{2}\oplus V_{6}
\]
\[
r_{5}\simeq V_{3}
\]
\[
s=\left(\frac{1}{2};4\right),\frac{5}{6},(1;2),\frac{7}{6},\left(\frac{3}{2};4\right),2,\left(\frac{5}{2};3\right),\left(\frac{7}{2};2\right),\frac{9}{2},\frac{11}{2}
\]

\subsubsection{$P_{6}$}

\[
r_{1}\simeq V_{2}\oplus V_{4}\oplus V_{6}\oplus V_{8}\oplus V_{10}\oplus V_{12}
\]
\[
r_{2}\simeq V_{2}\oplus V_{6}\oplus V_{8}\oplus V_{10}
\]
\[
r_{3}\simeq V_{4}\oplus V_{10}
\]
\[
r_{4}\simeq V_{2}
\]
\[
s=\frac{1}{2},\left(1;2\right),\left(2;3\right),\frac{5}{2},(3;2),4,5,6,7
\]

\subsubsection{$P_{7}$}

\[
r_{1}\simeq V_{1}\oplus V_{7}\oplus V_{9}\oplus V_{15}\oplus V_{17}
\]
\[
r_{2}\simeq V_{0}\oplus V_{8}\oplus V_{16}
\]
\[
r_{3}\simeq V_{1}
\]
\[
s=\left(\frac{1}{2};2\right),\frac{3}{2},\frac{5}{2},\left(\frac{9}{2};2\right),\frac{11}{2},\frac{17}{2},\frac{19}{2}
\]

\subsubsection{$P_{8}$}

\[
r_{1}\simeq V\oplus V_{17}\oplus V_{27}
\]
\[
r_{2}\simeq V_{0}
\]
\[
s=\frac{1}{2},\frac{11}{2},\frac{19}{2},\frac{29}{2}
\]

\subsection{Type $F_{4}$}

The Dynkin diagram with standard numbering is 

\begin{center} 
\texttt{
$F_4$: \centering \dynkin[label,edge length=1cm]F4 \newline 
}
\end {center}Recall that there is a relabelling of roots under $G\to{}^{L}G$:
$\alpha_{1}^{\vee}\rightsquigarrow\alpha_{4}$, $\alpha_{2}^{\vee}\rightsquigarrow\alpha_{3}$,
$\alpha_{3}^{\vee}\rightsquigarrow\alpha_{2}$, and $\alpha_{4}^{\vee}\rightsquigarrow\alpha_{1}$.
If $P_{i}=N_{i}\rtimes M_{i}$ corresponds to $\alpha_{i}$, then
\[
^{L}\mathfrak{n}_{1}\simeq\mathfrak{n}_{4}\quad{}^{L}\mathfrak{n}_{2}\simeq\mathfrak{n}_{3},\quad{}^{L}\mathfrak{n}_{3}\simeq\mathfrak{n}_{2},\ \text{ and }\ {}^{L}\mathfrak{n}_{4}\simeq\mathfrak{n}_{1}
\]
under the action of the corresponding $SL_{2}$. 

\subsubsection{$P_{1}$}

\[
r_{1}\simeq V_{0}\oplus V_{6}
\]
\[
r_{2}\simeq V_{6}
\]
\[
s=1,2,4
\]

\subsubsection{$P_{2}$}

\[
r_{1}\simeq V_{1}\oplus V_{3}
\]
\[
r_{2}\simeq V_{0}\oplus V_{2}\oplus V_{4}
\]
\[
r_{3}\simeq V_{1},\quad r_{4}=V_{2}
\]
\[
s=\left(\frac{1}{2};3\right),1,\left(\frac{3}{2};2\right),\frac{5}{2}
\]

\subsubsection{$P_{3}$}

\[
r_{1}=V_{1}\oplus V_{3}\oplus V_{5}
\]
\[
r_{2}=V_{0}\oplus V_{4}
\]
\[
r_{3}=V_{1}
\]
\[
s=\left(\frac{1}{2};2\right),\left(\frac{3}{2};2\right),\frac{5}{2},\frac{7}{2}
\]

\subsubsection{$P_{4}$}

\[
r_{1}=V_{3}\oplus V_{9}
\]
\[
r_{2}=V_{0}
\]
\[
s=\frac{1}{2},\frac{5}{2},\frac{11}{2}
\]

\newpage{}

\printbibliography

\end{document}